\documentclass[twoside,a4paper,reqno,11pt]{amsart} 
\usepackage[top=28mm,right=30mm,bottom=28mm,left=30mm]{geometry}

\usepackage{amsfonts, amsmath, amssymb, mathrsfs, bm, latexsym, stmaryrd, array, hyperref, mathtools, bold-extra, xcolor}

\renewcommand{\a}{\alpha}
\renewcommand{\b}{\beta}

\newcommand{\normeq}{\trianglelefteqslant}

\newcommand{\la}{\langle}
\newcommand{\ra}{\rangle}

\renewcommand{\to}{\rightarrow}

\newcommand{\leqs}{\leqslant}
\newcommand{\geqs}{\geqslant}

\newcommand{\Aut}{\mathrm{Aut}}

\newcommand{\Out}{\mathrm{Out}}
\newcommand{\soc}{\mathrm{soc}}
\newcommand{\PSL}{\mathrm{PSL}}
\newcommand{\GL}{\mathrm{GL}}
\newcommand{\PGL}{\mathrm{PGL}}
\newcommand{\SL}{\mathrm{SL}}

\newcommand{\PGammaL}{\mathrm{P\Gamma L}}

\newcommand{\PSp}{\mathrm{PSp}}
\newcommand{\Sp}{\mathrm{Sp}}
\newcommand{\PSU}{\mathrm{PSU}}

\newcommand{\SU}{\mathrm{SU}}
\newcommand{\PGU}{\mathrm{PGU}}

\newcommand{\SigmaL}{\mathrm{\Sigma L}}
\newcommand{\GU}{\mathrm{GU}}

\def\Ma{{\rm M}}
\def\Sym{\mathrm{Sym}}
\def\calB{\mathcal{B}}
\def\bbZ{\mathbb{Z}}

\def\GammaL{\mathrm{\Gamma L}}

\makeatletter
\newcommand{\imod}[1]{\allowbreak\mkern4mu({\operator@font mod}\,\,#1)}
\makeatother

\newtheorem{theorem}{Theorem} 
\newtheorem*{conj*}{Conjecture}

\newtheorem{corol}[theorem]{Corollary}
\newtheorem{prob}[theorem]{Problem}

\newtheorem{thm}{Theorem}[section] 
\newtheorem{prop}[thm]{Proposition} 
\newtheorem{lem}[thm]{Lemma}
\newtheorem{cor}[thm]{Corollary}

\newtheorem{hypo}[thm]{Hypothesis}
\newtheorem*{prob*}{Problem}

\theoremstyle{definition}
\newtheorem{rem}[thm]{Remark}

\newtheorem{ex}[thm]{Example}

\begin{document}
	
	\title[On finite permutation groups of rank three]{On finite permutation groups of rank three} 
	\author{Hong Yi Huang}
	\address{H.Y. Huang, School of Mathematics, University of Bristol, Bristol BS8 1UG, UK}
	\email{hy.huang@bristol.ac.uk}
	\author{Cai Heng Li}
	\address{C.H. Li, SUSTech International Center for Mathematics, and Department of Mathematics, Southern University of Science and Technology, Shenzhen 518055, Guangdong, P. R. China}
	\email{lich@sustech.edu.cn}
	\author{Yan Zhou Zhu}
	\address{Y.Z. Zhu, Department of Mathematics, Southern University of Science and Technology, Shenzhen 518055, Guangdong, P. R. China}
	\email{zhuyz@mail.sustech.edu.cn}
	\date{\today}
	
	\maketitle
	
	\begin{abstract}
		The classification of the finite primitive permutation groups of rank $3$ was completed in the 1980s and this landmark achievement has found a wide range of applications. In the general transitive setting, a classical result of Higman shows that every finite imprimitive rank $3$ permutation group $G$ has a unique non-trivial block system $\calB$ and this provides a natural way to partition the analysis of these groups. Indeed, the induced permutation group $G^{\calB}$ is $2$-transitive and one can also show that the action induced on each block in $\calB$ is also $2$-transitive (and so both induced groups are either affine or almost simple). In this paper, we make progress towards a classification of the rank $3$ imprimitive groups by studying the case where the induced action of $G$ on a block in $\calB$ is of affine type. Our main theorem divides these rank $3$ groups into four classes, which are defined in terms of the kernel of the action of $G$ on $\calB$. In particular, we completely determine the rank $3$ semiprimitive groups for which $G^\calB$ is almost simple, extending recent work of Baykalov, Devillers and Praeger. We also prove that if $G$ is rank $3$ semiprimitive and $G^\calB$ is affine, then $G$ has a regular normal subgroup which is a special $p$-group for some prime $p$.
	\end{abstract}
	
	\section{Introduction}
	
	\label{s:intro}

	Let $G\leqs\Sym(\Omega)$ be a permutation group on a finite set $\Omega$ and let $H$ be a point stabiliser. Recall that the \textit{rank} of $G$ is the number of orbits in the componentwise action of $G$ on $\Omega\times\Omega$. Equivalently, if $G$ is transitive with a point stabiliser $H$, then the rank of $G$ is the number of $H$-orbits on $\Omega$. The study of rank $3$ groups dates back to work of Higman \cite{H_rank3} in the 1960s and they continue to find a wide range of connections and applications to other areas of algebra and combinatorics. This includes transitive graph decompositions \cite{BPP_trans-decomp,PP_trans-decomp}, partial linear spaces \cite{BDFP_partial,BDP_innately,D_partial1,D_partial2} and combinatorial designs \cite{BMF_des,D_sym_des1,D_sym_des2,D_sym_des3,M_des}. 
	
	Clearly, every permutation group of rank $3$ is transitive, and recall that a transitive permutation group is \textit{primitive} if a point stabiliser is a maximal subgroup, and \textit{imprimitive} otherwise. The classification of the rank $3$ primitive groups was completed by Liebeck and Saxl in the 1980s \cite{L_rank3,LS_rank3}, building on earlier work by Bannai, Kantor and Liebler \cite{B_rank3,KL_rank3}. It is natural to seek extensions of this classification to interesting families of imprimitive groups. The first step in this direction was taken by Devillers et al. in \cite{DGLPP_rank3}, where the finite quasiprimitive groups of rank $3$ are determined (recall that a transitive group is called \textit{quasiprimitive} if every non-trivial normal subgroup is transitive). Recently, this was further extended to the family of innately transitive groups \cite{BDP_innately}, where we recall that a permutation group is  \textit{innately transitive} if it has a transitive minimal normal subgroup (see  Theorem \ref{t:innately} for further details).
	
	As introduced in \cite{BM_semi}, a transitive permutation group is said to be \textit{semiprimitive} if each normal subgroup is either transitive or semiregular. In particular, note that every innately transitive group is semiprimitive. The class of semiprimitive groups includes some interesting families of permutation groups. For example, every finite Frobenius group is semiprimitive, and the automorphism group of every graph that is vertex-transitive and locally quasiprimitive is also semiprimitive (see \cite[Lemma 2.1]{BM_semi} and \cite[Lemma 8.1]{GM_semi}, respectively). The study of semiprimitive groups is partly motivated by its applications to collapsing transformation monoids \cite{BM_semi} and the graph-restrictive problem for permutation groups \cite{GM_semi}, and we refer the reader to \cite{GM_semi} for a description of the structure of semiprimitive groups.
	
	Our first theorem, which will be proved in Section \ref{s:semi_proof}, is a key step towards a classification of semiprimitive rank $3$ groups, extending earlier results \cite{BDP_innately,DGLPP_rank3}. A significant motivation here is Proposition \ref{p:red}, which shows that the class of semiprimitive rank $3$ groups arises naturally in the study of general rank $3$ transitive groups, and this will be an important ingredient in the proof of our main result (see Theorem \ref{thm:reduction} below). One of the main difficulties in proving Theorem \ref{thm:semi} is in dealing with non-split extensions of almost simple groups by groups of prime order, and the groups with a regular normal subgroup, noting that neither case arises in the analysis of rank $3$ innately transitive groups. This requires us to adopt different techniques from those required in the innately transitive setting. For example, we refer the reader to Lemma \ref{l:semi} for basic observations of semiprimitive rank $3$ groups which are not innately transitive.
	
	In order to state the result, recall that a $p$-group $P$ is \textit{special} if either $P$ is elementary abelian, or $P' = Z(P) = \Phi(P)$ is elementary abelian, where $\Phi(P)$ is the Frattini subgroup of $P$. We also adopt the standard notation of classical groups, so a general semilinear group is denoted $\GammaL_d(q)$ for some $d$ and $q$, and we write $\PGammaL_d(q):= \GammaL_d(q)/Z(\GL_d(q))$. Recall that for a prime $p$ and an integer $f$, a prime divisor $r$ of $p^f-1$ is called a \textit{primitive prime divisor} if $r$ does not divide $p^m-1$ for any positive integer $m < f$.
	
	\begin{theorem}
		\label{thm:semi}
		Let $G\leqs\Sym(\Omega)$ be a finite semiprimitive permutation group of rank $3$ with point stabiliser $H$ and assume that $G$ is not innately transitive. Then up to permutation isomorphism, one of the following holds:
		\begin{itemize}\addtolength{\itemsep}{0.2\baselineskip}
			\item[{\rm (a)}] $(G,H) = (3.S_6,S_5)$ or $(2.\mathrm{M}_{12},\mathrm{M}_{11})$.
			\item[{\rm (b)}] $\Omega$ is the set of $C$-orbits on $\mathbb{F}_q^d\setminus\{0\}$ for $\bbZ_{(q-1)/r}\cong C < Z(\GL_d(q))$, and we have
			\[
			r.\PSL_d(q)\cong \SL_d(q)C/ C\leqs G\leqs \GammaL_d(q)/C\cong r.\PGammaL_d(q),
			\]
			where $(d,q)\ne (2,3)$, $q = p^{f}$, $p$ is a prime and $r$ is a primitive prime divisor of $p^{r-1}-1$.
			\item[{\rm (c)}] $G$ has a regular normal subgroup $N$, where $N$ is a special $p$-group for some prime $p$.
		\end{itemize}
	\end{theorem}
	
	We remark that the groups arising in part (a) have rank $3$. Similarly, we can determine precise necessary and sufficient conditions for the groups in (b) to have rank $3$ (see Proposition \ref{p:PSL}) and we note that there are genuine examples.
	
	Let us comment on part (c) of Theorem \ref{thm:semi}. Let $G$ be an arbitrary rank $3$ permutation group with a regular normal subgroup $N$. First note that  $N\normeq G\leqs N{:}\Aut(N)$, and  $\Aut(N)$ has at most $3$ orbits on $N$. A key result needed for part (c) is Lemma~\ref{l:regularnormal}, which asserts that if $N$ is a finite group and $\Aut(N)$ has at most $3$ orbits on $N$, then it is isomorphic to one of the following groups (here $p$ and $q$ are distinct primes):
	\begin{itemize}\addtolength{\itemsep}{0.2\baselineskip}
		\item[{\rm (i)}] an elementary abelian $p$-group;
		\item[{\rm (ii)}] a Frobenius $\{p,q\}$-group;
		\item[{\rm (iii)}] a homocyclic group $\bbZ_{p^2}^k$ for some $k$;
		\item[{\rm (iv)}] a special $2$-group of exponent $4$;
		\item[{\rm (v)}] a non-abelian special $p$-group of exponent $p\geqs 3$.
	\end{itemize}
	Under the additional assumption that $G$ is semiprimitive, we will prove that $N$ is not of type (ii) or (iii) in the above list of possibilities. For example, Lemma \ref{l:p=r_affine} shows that (ii) does not occur. It turns out that there is a genuine rank $3$ semiprimitive group with $N$ of type (i); we refer the reader to Example~\ref{ex:affine}. We will return to the classification of semiprimitive rank $3$ groups arising in part (c) in future work.
	
	Next, let us turn to the more ambitious goal of determining all the finite permutation groups of rank $3$. In this setting, a classical result of Higman \cite{H_rank3} states that if $G\leqs\mathrm{Sym}(\Omega)$ is imprimitive of rank $3$, then $G$ has a unique non-trivial block system $\calB$, and the induced permutation group $G^\calB$ on $\calB$ is $2$-transitive (and is therefore affine or almost simple). The first systematic study of imprimitive rank $3$ permutation groups was initiated by Devillers et al. in \cite{DGLPP_rank3}. For example, if $G_B$ denotes the setwise stabiliser of a block $B\in\calB$, then \cite[Lemma 2.1]{DGLPP_rank3} shows that the induced permutation group $G_B^B$ is $2$-transitive. In the same paper, the authors classify the rank $3$ groups $G$ under the assumption that $G_B^B$ is almost simple; see Theorem \ref{t:DGLPP_1} (recall that a group $X$ is \textit{almost simple} if its \textit{socle} $\soc(X)$ is non-abelian simple). In view of these results, it is worth noting that $G_B^B$ is affine for all the groups arising in Theorem \ref{thm:semi}. In addition, $G^\calB$ is almost simple if $G$ is one of the groups listed in cases (a) and (b), whereas $G^\calB$ is affine in case (c).
	
	The main result of this paper is Theorem \ref{thm:reduction} below. This  extends Theorem \ref{thm:semi} by treating the general case where $G_B^B$ is affine. In the statement, we write $K=G_{(\calB)}$ for the kernel of the action of $G$ on $\calB$. 
	
	\begin{theorem}\label{thm:reduction}
		Let $G\leqs\Sym(\Omega)$ be a finite transitive rank $3$  permutation group with non-trivial block system $\calB$. Let $B,B' \in \calB$ be distinct blocks and assume $G_B^B$ is affine. Then one of the following holds:
		\begin{itemize}\addtolength{\itemsep}{0.2\baselineskip}
			\item[{\rm (A)}] $G$ is either innately transitive, or it is permutation isomorphic to a group recorded in case {\rm (a)} or {\rm (b)} of Theorem \ref{thm:semi};
			\item[{\rm (B)}] $N\normeq G\leqs N{:}\Aut(N)$, where $N$ is a regular normal subgroup and $\Aut(N)$ has at most $3$ orbits on $N$;
			\item[{\rm (C)}] $K_{(B)}$ is transitive on $B'$;
			\item[{\rm (D)}] $K_{(B)}\ne 1$ is intransitive on $B'$, and $G$ has an elementary abelian self-centralising normal subgroup.
		\end{itemize}
	\end{theorem}
	
	We will use Theorem \ref{thm:semi} to establish Theorem \ref{thm:reduction} in Section \ref{s:reduction}, providing an independent proof of Theorem \ref{thm:semi} in Section \ref{s:semi_proof}.
	
	As noted above, the possibilities for the group $N$ in Theorem \ref{thm:reduction}(B) are described in Lemma \ref{l:regularnormal} and they are completely determined in \cite{Z_3orbits}. So in order to classify the rank $3$ groups arising in Theorem \ref{thm:reduction}(B), it suffices to determine the proper subgroups of $\Aut(N)$ with exactly $3$ orbits on $N$.
	
	A full classification of the finite transitive permutation groups of rank $3$ remains out of reach and this seems to be a very difficult problem (for example, see the discussion in \cite{BDP_innately,DGLPP_rank3}). However, Theorem \ref{thm:reduction} provides a reduction of the general problem to two special classes, as described in the following corollary.
	
	\begin{corol}
		\label{cor:remaining}
		Let $G\leqs\mathrm{Sym}(\Omega)$ be a finite permutation group of rank $3$. Then one of the following holds:
		\begin{itemize}\addtolength{\itemsep}{0.2\baselineskip}
			\item[{\rm (i)}] $G$ lies in class {\rm (C)} or {\rm (D)} of Theorem \ref{thm:reduction};
			\item[{\rm (ii)}] $G$ is in a known list;
			\item[{\rm (iii)}] $N\normeq G\leqs N{:}\Aut(N)$ with $N$ given in \cite{Z_3orbits}.
		\end{itemize}
	\end{corol}
	
	We note that any permutation group $G$ satisfying the condition in 
	Theorem \ref{thm:reduction}(C) has rank $3$ (see Corollary \ref{c:KB_trans}). This includes the groups with $\soc(G_B^B)^{|\calB|}\leqs G$, and we refer the reader to Example \ref{ex:trans} for an infinite family of groups that arise here with the additional property that $\soc(G_B^B)^{|\calB|}$ is not a subgroup of $G$.
	
	We have been unable to identify a genuine rank $3$ group satisfying the conditions in Theorem \ref{thm:reduction}(D), so we pose the following open problem.
	
	\begin{prob}
		Is there a rank $3$ group satisfying the conditions in Theorem \ref{thm:reduction}(D)?
	\end{prob}
	
	Further properties of such a group (should one exist) will be given in Lemma \ref{l:KB>1_intrans}. In particular, it is shown that $K$ is a Frobenius group with cyclic complement.

	%

	\subsection*{Notation}
	
	For a positive integer $n$ and a prime number $p$, we write $(n)_p$ for the highest power of $p$ dividing $n$. Suppose $G$ is a group acting on a set $\Delta$ and $\Sigma\subseteq \Delta$. Then we write $G_{(\Sigma)}$ and $G_\Sigma$ for the pointwise and setwise stabilisers of $\Sigma$ in $G$, respectively. If $\Sigma = \{\alpha,\beta\}$ then we simply write $G_{\alpha,\beta}$ for the pointwise stabiliser $G_\a\cap G_\b$. The induced permutation group of $G$ on $\Delta$ is denoted $G^\Delta$, so $G^\Delta \cong G/G_{(\Delta)}$. We also use $G^\infty$ for the unique perfect group appearing in the derived series of $G$, and we adopt the standard notation for simple groups of Lie type from \cite{KL_classical}.

	\subsection*{Acknowledgment}
	
	This work was partially supported by NNSFC grant no. 11931005. The first author thanks the China Scholarship Council for supporting his doctoral studies at the University of Bristol. He also thanks the Southern
	University of Science and Technology (SUSTech) for their generous hospitality during a
	visit in 2023.

	\section{Preliminaries}
	
	\label{s:pre}

	\subsection{The structure of imprimitive groups of rank $3$}
	
	\label{ss:pre_struc}
	
	In this section, we record some preliminary results on the structure of imprimitive groups of rank $3$, most of which are taken from \cite{DGLPP_rank3}. As shown in \cite[Lemma 3.1]{DGLPP_rank3}, any permutation group of rank $3$ is permutation isomorphic to a group $G$ satisfying the following hypothesis.
	
	\begin{hypo}
		\label{hypo:imprim}
		Let $Y$ and $X$ be $2$-transitive permutation groups on a finite set $B$ and $\{1,\dots,n\}$, respectively. Assume $G\leqs Y\wr X$ acts imprimitively on $\Omega = B\times \{1,\dots,n\}$ with associated non-trivial block system $\calB$, where $G$ induces $Y$ on $B$ and $X$ on $\{1,\dots,n\}$.
	\end{hypo}
	
	More specifically, we have $\calB = \{(B,1),\dots,(B,n)\}$, where $G^\calB\cong X$ and $G_{(B,1)}^{(B,1)}\cong Y$. Without loss of generality, we may identify $(B,1)$ with $B$ and $\calB = B^G$, so we may write $Y = G_B^B$ from now on. The following lemmas are \cite[Lemmas 2.3 and 4.2]{DGLPP_rank3}, respectively.
	
	\begin{lem}
		\label{l:DGLPP:l:2.3}
		Assume Hypothesis \ref{hypo:imprim}, where $G$ has rank $3$ on $\Omega$. Then $\calB$ is the unique non-trivial block system of $G$.
	\end{lem}
	
	\begin{lem}
		\label{l:DGLPP:l:4.2}
		Assume Hypothesis \ref{hypo:imprim}. Let $B'\in\calB\setminus\{B\}$, $\beta\in B$ and $\beta'\in B'$. Then $G$ has rank $3$ if and only if the following conditions hold:
		\begin{itemize}\addtolength{\itemsep}{0.2\baselineskip}
			\item[{\rm (i)}] $G_{B,B'}$ acts transitively on $B\times B'$; and
			\item[{\rm (ii)}] $G_\beta$ is transitive on $\calB\setminus\{B\}$.
		\end{itemize}
		Moreover, if $G$ has rank $3$, then $G_{\beta,\beta'}$ is a subgroup of index $|B|^2$ in $G_{B,B'}$.
	\end{lem}
	
	Note that the kernel of $G$ acting on $\calB$ is
	\begin{equation*}
	K:=G_{(\calB)} = G\cap Y^n.
	\end{equation*}
	
	\begin{cor}
		\label{c:KB_trans}
		Assume Hypothesis \ref{hypo:imprim} and let $B'\in\calB\setminus\{B\}$. Then $G$ has rank $3$ on $\Omega$ if $K_{(B)}$ is transitive on $B'$.
	\end{cor}
	
	\begin{proof}
		Assume that $K_{(B)}$ is transitive on $B'$. Note that $1\ne K^B\normeq G_B^B = Y$, and thus $K$ acts transitively on $B$. It follows that $K$ is transitive on $B\times B'$. In view of Lemma \ref{l:DGLPP:l:4.2}, it suffices to show that $G_\beta$ is transitive on $\calB\setminus\{B\}$, where $\beta\in B$. Let $B''\in\calB\setminus\{B\}$. Since $G$ is $2$-transitive on $\calB$, there exists $x\in G_B$ such that $(B')^x = B''$. The transitivity of $K$ on $B$ implies that there exists $y\in K$ such that $\beta^{xy} = \beta$. Hence, $xy\in G_\beta$ and
		\begin{equation*}
		(B')^{xy} = (B'')^y = B''
		\end{equation*}
		as $K = G_{(\calB)}$. This completes the proof.
	\end{proof}

	\begin{lem}
		\label{l:norm_trans}
		Assume Hypothesis \ref{hypo:imprim}, where $G$ has rank $3$ on $\Omega$. Let $N$ be a normal subgroup of $G$. Then either $N\leqs K$, or $N$ is transitive on $\Omega$.
	\end{lem}
	
	\begin{proof}
		Suppose $1\ne N\normeq G$ is intransitive on $\Omega$. Note that an $N$-orbit on $\Omega$ is a block of $G$. Hence, the set of $N$-orbits on $\Omega$ is exactly $\calB$ by Lemma \ref{l:DGLPP:l:2.3}. This implies that $N\leqs G_{(\calB)} = K$.
	\end{proof}
	
	Note that if a rank $3$ group $G$ has a regular normal subgroup $N$, then $N\normeq G\leqs N{:}\Aut(N)$, which implies that $\Aut(N)$ has at most $3$ orbits on $N$. Conversely, if $G = N{:}\Aut(N)$ is the holomorph of a group $N$, where $\Aut(N)$ has exactly $3$ orbits on $N$, then $G$ has rank $3$ with its natural action on $N$.
	
	\begin{lem}
		\label{l:regularnormal}
		Suppose $N\neq 1$ is a finite group such that $\Aut(N)$ has at most $3$ orbits on $N$.
		Then $N$ is isomorphic to one of the following groups, where $p$ and $q$ are distinct primes:
		\begin{itemize}\addtolength{\itemsep}{0.2\baselineskip}
			\item[{\rm (i)}] an elementary abelian $p$-group;
			\item[{\rm (ii)}] a Frobenius $\{p,q\}$-group with a Frobenius complement isomorphic to $\bbZ_q$; 
			\item[{\rm (iii)}] a homocyclic group $\bbZ_{p^2}^k$ for some $k$;
			\item[{\rm (iv)}] a special $2$-group of exponent $4$;
			\item[{\rm (v)}] a non-abelian special $p$-group of exponent $p$ with $p\geqs 3$.
		\end{itemize}
	\end{lem}
	\begin{proof}
		First note that if $\Aut(N)$ has exactly $2$ orbits on $N$, then $N$ is clearly an elementary abelian $p$-group for some prime $p$, as in~(i).
		Now assume that $\Aut(N)$ has exactly $3$ orbits on $N$.
		It is easy to see that $N$ is either a $p$-group or a $\{p,q\}$-group for primes $p$ and $q$.
		We remark that the $\{p,q\}$-groups with three automorphism orbits have been classified in~\cite[Theorem 2]{LM_aut}, and such a group $N$ is a Frobenius group as recorded in~(ii).
		
		To complete the proof, assume $N$ is a $p$-group.
		Since $\Aut(N)$ has at most $3$ orbits on $N$, the exponent of $N$ is at most $p^2$.
		If $N$ is abelian, we may write $N\cong\bbZ_{p^2}^k\times \bbZ_p^\ell$ for some $k\geqs 1$ and $\ell\geqs 0$. 
		Note that if $\ell\geqs 1$, then $N$ has at least two $\Aut(N)$-classes of elements of order $p$, and thus $N\cong \bbZ_{p^2}^k$ as in~(iii).
		
		Finally, assume $N$ is a non-abelian $p$-group and $\Aut(N)$ has exactly $3$ orbits on $N$. Then $N$ has a unique non-trivial proper characteristic subgroup.
		In particular, we have $\Phi(N)=N'=Z(N)$, where $\Phi(N)$ is the Frattini subgroup of $N$, and hence $N$ is a special $p$-group.
		If $p = 2$ then $N$ has exponent $4$, since a group of exponent $2$ is abelian. This is recorded as (iv).
		Suppose that $p$ is odd and $N$ has exponent $p^2$.
		Then $\Aut(N)$ acts transitively on the elements of order $p$.
		However, \cite[Corollary 4]{S_aut} implies that $N$ is abelian, which is a contradiction.
		Thus, $N$ is a special $p$-group of exponent $p$ when $p$ is odd, which gives (v).
	\end{proof}

As remarked in Section \ref{s:intro}, a much stronger result in \cite{Z_3orbits} completely classifies the finite groups $N$ with precisely $3$ orbits under $\Aut(N)$. To establish the results in this paper, Lemma \ref{l:regularnormal} is sufficient.

	To conclude this subsection, we record some existing results in the literature stepping towards a classification of rank $3$ imprimitive permutation groups. The following is \cite[Theorem 1]{DGLPP_rank3}.
	
	\begin{thm}
		\label{t:DGLPP_1}
		Assume Hypothesis \ref{hypo:imprim}, where $G_B^B$ is almost simple. Then $G$ has rank $3$ on $\Omega$ if and only if one of the following holds:
		\begin{itemize}\addtolength{\itemsep}{0.2\baselineskip}
			\item[{\rm (i)}] $\soc(Y)^n\leqs G$;
			\item[{\rm (ii)}] $G$ is quasiprimitive and rank $3$ on $\Omega$;
			\item[{\rm (iii)}] $n = 2$, $|B|=6$ and $G\in\{\Ma_{10},\PGL_2(9),\PGammaL_2(9)\}$;
			\item[{\rm (iv)}] $n = 2$, $|B| = 12$ and $G = \Ma_{12}.2$.
		\end{itemize}
	\end{thm}
	
	Note that $K = 1$ if and only if $G$ is quasiprimitive, and the rank $3$ quasiprimitive imprimitive groups are recorded in \cite[Table 1]{DGLPP_rank3} by \cite[Theorem 2]{DGLPP_rank3}. A broader class of permutation groups, so-called \textit{innately transitive} groups, consists of the groups with a transitive minimal normal subgroup. Recently, the rank $3$ innately transitive non-quasiprimitive groups are determined as in \cite[Table 3]{BDP_innately} by \cite[Theorem C]{BDP_innately}.

	\begin{thm}
		\label{t:innately}
		Let $G\leqs\Sym(\Omega)$ be an innately transitive imprimitive permutation group. Then $G$ has rank $3$ if and only if $G$ is one of the groups recorded in \cite[Table 3]{BDP_innately} or \cite[Table 1]{DGLPP_rank3}.
	\end{thm}

	\subsection{$2$-transitive permutation groups}
	
	\label{ss:2-trans}
	
	As discussed above, any permutation group of rank $3$ is permutation isomorphic to a group $G$ satisfying Hypothesis \ref{hypo:imprim}, so the classification of the finite $2$-transitive permutation groups, which has been done in \cite{He_2-trans,Hu_2-trans}, plays a key role in the study of imprimitive groups of rank $3$. Observe that a $2$-transitive permutation group is either almost simple or affine (this observation dates back to Burnside \cite{B_finite_gp}), and its classification relies on the classification of finite simple groups. Here we record the classification of $2$-transitive permutation groups as follows.
	
	\begin{thm}
		\label{t:2-trans_as}
		Let $G$ be an almost simple permutation group with socle $G_0$ and point stabiliser $H$. Then $G$ is $2$-transitive if and only if $G$ is one of the groups listed in Table \ref{tab:2-trans_as}.
	\end{thm}

	\begin{table}[ht]
		\[
		\begin{array}{llll} \hline
		G_0 &  H\cap G_0 & m(G_0) & \mbox{Conditions}
		\\ \hline
		A_m & A_{m-1} & \bbZ_2 & m\ne 6,7\\
		&  & \bbZ_6 & \mbox{$m=6,7$, $G\leqs S_m$}\\
		\PSL_d(q) & [q^{d-1}]{:}\GL_{d-1}(q)/(d,q-1) & \mbox{see Remark \ref{r:tab_as}} & G\leqs \PGammaL_d(q)\\
		\Sp_{2d}(2) & \mathrm{SO}_{2d}^+(2) & 1 & d\geqs 4\\
		&  & \bbZ_2 & d = 3\\
		\Sp_{2d}(2) & \mathrm{SO}_{2d}^-(2) & 1 & d\geqs 4\\
		&  & \bbZ_2 & d = 3\\
		\PSU_3(q) & [q^3]{:}\bbZ_{(q^2-1)/(3,q+1)} & \bbZ_{(3,q+1)} & q\geqs 3\\
		{^2}B_2(q) & [q^2]{:}\bbZ_{q-1} & 1 & q = 2^{2d+1}>8\\
		&  & \bbZ_2^2 & q = 8\\
		{^2}G_2(q) & [q^3]{:}\bbZ_{q-1} & 1 & q = 3^{2d+1}>3\\
		A_7 & \PSL_2(7) & \bbZ_6 & G = G_0\\
		\PSL_2(8) & D_{18} & 1 &G = G_0.3\\
		\PSL_2(11) & A_5 & \bbZ_2 &G = G_0 \\
		\mathrm{M}_{11} & \mathrm{M}_{10} & 1 &\\
		\mathrm{M}_{11} & \PSL_2(11) & 1 &\\
		\mathrm{M}_{12} & \mathrm{M}_{11} & \bbZ_2 & G = G_0 \\
		\mathrm{M}_{22} & \PSL_3(4) & \bbZ_{12} & \\
		\mathrm{M}_{23} & \mathrm{M}_{22} & 1 & \\
		\mathrm{M}_{24} & \mathrm{M}_{23} & 1 & \\
		\mathrm{HS} & \PSU_3(5).2 & \bbZ_2 & G = G_0\\
		\mathrm{Co}_3 & \mathrm{McL}.2 & 1 &
		\\ \hline
		\end{array}
		\]
		\caption{Almost simple $2$-transitive groups}
		\label{tab:2-trans_as}
	\end{table}
	
	\begin{rem}
		\label{r:tab_as}
		In the third column of Table \ref{tab:2-trans_as}, we record the Schur multiplier $m(G_0)$ of the simple group $G_0$ (recall that a group $\widetilde{G_0}$ is called a \textit{covering group} of $G_0$ if $Z(\widetilde{G_0})\leqs \widetilde{G_0}'$ and $\widetilde{G_0}/Z(\widetilde{G_0}) \cong G_0$, and the \textit{Schur multiplier} of $G_0$ is the centre of its (unique) maximal covering group). This variant will be used later in Section \ref{s:semi_proof}. When $G_0 = \PSL_d(q)$, we have $m(G_0) = \bbZ_{(d,q-1)}$ unless $(d,q) = (2,4),(3,2),(4,2)$ and $m(G_0) = \bbZ_2$, or $(d,q,m(G_0)) = (2,9,\bbZ_6)$ or $(3,4,3\times \bbZ_4^2)$.
	\end{rem}
	
	Note that if $G = V{:}T\leqs \mathrm{AGL}_d(p)$ is affine, where $V = \bbZ_p^d$ is elementary abelian, then $G$ is $2$-transitive if and only if $T\leqs\GL_d(p)$ is a transitive linear group (that is, $T$ acts transitively on the non-zero vectors of $\mathbb{F}_p^d$).

	\begin{thm}
		\label{t:2-trans_affine}
		Let $T\leqs\GL_d(p)$ be a transitive linear group. Then $T$ is one of the groups listed in Table \ref{tab:2-trans_affine}.
	\end{thm}

	\begin{table}[ht]
		\[
		\begin{array}{lll} \hline
		p^d & T & \mbox{Condition}
		\\ \hline
		q^a & \SL_a(q)\normeq T \leqs \mathrm{\Gamma L}_a(q)& \\
		q^{2a} & \Sp_{2a}(q)'\normeq T &a\geqs 2\\
		q^6 & G_2(q)'\normeq T & \mbox{$p = 2$}\\
		p^2 & \SL_2(3)\normeq T\leqs \SL_2(3).\bbZ_{p-1} & p = 5,7,11,23\\
		p^2 & \SL_2(5)\normeq T\leqs \SL_2(5).\bbZ_{(p-1)/2} & p = 11,19,29,59\\
		2^4 & A_7 &\\
		3^4 & \SL_2(5)\normeq T \leqs \SL_2(5).D_8&\\
		3^4 & 2^{1+4}_-\normeq T \leqs 2_-^{1+4}.S_5&\\
		3^6 & \SL_2(13) & 
		\\ \hline
		\end{array}
		\]
		\caption{Transitive linear groups}
		\label{tab:2-trans_affine}
	\end{table}

	\begin{cor}
		\label{c:trans_linear_prime-power}
		Suppose $T\leqs\GL_d(p)$ is a transitive linear group. If $p^d$ divides $|T|$, then one of the following holds:
		\begin{itemize}\addtolength{\itemsep}{0.2\baselineskip}
			\item[{\rm (i)}] $\SL_a(p^f)\normeq T\leqs \mathrm{\Gamma L}_a(p^f)$ for $a\geqs 3$ and $d = af$;
			\item[{\rm (ii)}] $\Sp_{2a}(p^f)'\normeq T$ for $a\geqs 2$ and $d = 2af$;
			\item[{\rm (iii)}] $G_2(2^f)'\normeq T$ and $(p,d) = (2,6f)$.
		\end{itemize}
	\end{cor}

	In the remainder of this section, we will give some properties of transitive linear groups, which will be useful in the proof of Theorem \ref{thm:semi}. To start with, let us record a theorem of Guralnick \cite{G_prime-power}, which determines the subgroups of finite simple groups of prime-power index.
	
	\begin{thm}
		\label{t:prime-power}
		Let $G$ be a non-abelian finite simple group and let $r$ be a prime number. If $H < G$ and $|G:H| = r^m$ for some positive integer $m$, then one of the following holds:
		\begin{itemize}\addtolength{\itemsep}{0.2\baselineskip}
			\item[{\rm (i)}] $G = A_n$ and $H \cong A_{n-1}$ with $n = r^m$;
			\item[{\rm (ii)}] $G = \PSL_n(q)$, $H$ is of type $P_1$, and $|G:H| = (q^n-1)/(q-1) = r^m$;
			\item[{\rm (iii)}] $(G,H) = (\PSL_2(11),A_5)$, $(\Ma_{23},\Ma_{22})$ or $(\Ma_{11},\Ma_{10})$;
			\item[{\rm (iv)}] $G = \PSU_4(2)\cong\PSp_4(3)$ and $H$ is a subgroup of index $27$.
		\end{itemize}
	\end{thm}

	\begin{lem}
		\label{l:trans_linear_prime-power_index}
		Suppose $T\leqs\GL_d(p)$ is a transitive linear group with a subgroup of index $p^d$. Then $(d,p) = (3,2)$ and $T = \GL_3(2)$.
	\end{lem}
	
	\begin{proof}
		It suffices to consider the cases recorded in Corollary \ref{c:trans_linear_prime-power}. Note that the group $T^{\infty}$ is quasisimple in each case. It is also easy to see that $p\nmid |Z(T^{\infty})|$ and $p^d\nmid |T:T^{\infty}|$.
		
		Suppose $T$ has a subgroup of index $p^d$. With the observations above, the simple group $T^{\infty}/Z(T^{\infty})$ has a subgroup of index $p^m$ for some $m\geqs 1$. By Theorem \ref{t:prime-power}, the only possibilities for $(T^{\infty}/Z(T^{\infty}),p^m)$ are
		\begin{equation*}
		(\PSL_3(2),2^3),\ (\PSL_4(2),2^3),\ (\PSp_4(3),3^3).
		\end{equation*}
		Now the result follows with the aid of {\sc Magma} \cite{Magma}.
	\end{proof}
	
	We will also need to use the following lemma on perfect groups.
	
	\begin{lem}
		\label{l:perfect}
		Let $G$ be a group and let $N,M\normeq G$. Suppose $N$ is perfect and $N\cap M\leqs Z(N)$. Then $[N,M] = 1$.
	\end{lem}
	
	\begin{proof}
		Write $\overline{N}$ and $\overline{M}$ for the images of $N$ and $M$ in the quotient group $G/Z(N)$, respectively, noting that $[\overline{N},\overline{M}] = 1$. Let $g\in M$, which induces an automorphism $\a_g$ on $N$ by conjugation. Then the action of $\a_g$ on $\overline{N}$ is trivial as $[\overline{N},\overline{M}] = 1$. Note that the map $h\mapsto h^{\a_g}h^{-1}$ is a group homomorphism from $N$ to $Z(N)$, which implies that $\a_g$ is a trivial automorphism as $N$ is perfect. Therefore, we have $[N,M] = 1$.
	\end{proof}

	Recall that the \textit{soluble radical} of a group is its maximal soluble normal subgroup.
	
	\begin{lem}
		\label{l:sol_rad}
		Suppose $T$ is a transitive linear group and $T^{\infty}$ is quasisimple. Then $T$ has a cyclic soluble radical.
	\end{lem}
	
	\begin{proof}
		Let $R$ be the soluble radical of $T$. Then $T^{\infty}\cap R\leqs Z(T^{\infty})$ is a normal subgroup of $T$. By Lemma \ref{l:perfect}, we have $[R,T^{\infty}] = 1$. Now $T^{\infty}$ is irreducible on $\mathbb{F}_p^d$ by inspecting Table \ref{tab:2-trans_affine}. It can be deduced from Schur's lemma that $C_{\GL_d(p)}(T^{\infty})$ is isomorphic to the multiplicative group of a field, so $R\leqs C_{\GL_d(p)}(T^{\infty})$ is cyclic.
	\end{proof}
	
	We conclude this section with two applications of Lemma \ref{l:sol_rad}.
	
	\begin{lem}
		\label{l:norm_series}
		Suppose $T$ is a transitive linear group and $T^{\infty}$ is quasisimple. If $R\normeq S\normeq T$ and $S/R$ is an elementary abelian $r$-group for some prime $r$, then either $S/R\cong \bbZ_r$, or $T^{\infty}\leqs R$ and $T/R$ is metacyclic.
	\end{lem}
	
	\begin{proof}
		By Lemma \ref{l:sol_rad}, the soluble radical of $T$ is cyclic. Hence, $S$ is either cyclic or insoluble. If the former holds, then $S/R$ is also cyclic and thus $S/R\cong\bbZ_r$. 
		Now assume $S$ is insoluble.
		Since $S^\infty \normeq T^\infty$ and $T^\infty$ is quasisimple, we must have $S^\infty= T^\infty$.
		Since $S/R$ is elementary abelian, we have $T^{\infty}\leqs R$. By inspecting Table \ref{tab:2-trans_affine}, $T/T^{\infty}$ is metacyclic, and hence $T/R$ is also metacyclic.
	\end{proof}
	
	The following result might be of independent interest.
	
	\begin{thm}
		\label{t:norm_reducible}
		Suppose $T\leqs\GL_d(p)$ is a transitive linear group and $R\normeq T$ is reducible on $\mathbb{F}_p^d$. Then $R$ is cyclic.
	\end{thm}
	
	\begin{proof}
		First assume $T$ is insoluble. If $T^{\infty}$ is quasisimple, then $T^{\infty}$ is irreducible by inspecting Table \ref{tab:2-trans_affine}, and thus $R$ is contained in the soluble radical of $T$, so $R$ is cyclic by Lemma \ref{l:sol_rad}. If $T^{\infty}$ is not quasisimple, then it is easy to see that $(d,p) = (4,3)$ and $T\in\{2_-^{1+4}.A_5,2_-^{1+4}.S_5\}$ by inspecting Table \ref{tab:2-trans_affine}. In this setting, the only non-trivial reducible normal subgroup of $T$ is isomorphic to $\bbZ_2$.
		
		To complete the proof, we assume $T$ is soluble. Note that the only infinite family of soluble transitive linear groups is $T\leqs \mathrm{\Gamma L}_1(p^d)$, and the result for the other cases can be verified with the aid of {\sc Magma}. Thus, we may assume that
		\begin{equation*}
		T\leqs\mathrm{\Gamma L}_1(p^d) = \la x\ra{:}\la y\ra\cong \bbZ_{p^d-1}{:}\bbZ_d.
		\end{equation*}
		The groups with $(d,p) = (6,2)$ can be handled using {\sc Magma}. Now assume $d = 2$ and $p = 2^k-1$ for some $k$, and we may assume that $x^ay\in R$ for some $a$. If $T = \mathrm{\Gamma L}_1(p^d)$, then $[x,x^ay] = x^{p-1}\in R$, which gives a contradiction as $\la x^{p-1}\ra$ is irreducible. Thus, $T<\mathrm{\Gamma L}_1(p^d)$, noting that $|\mathrm{\Gamma L}_1(p^d):T| = 2$. In this setting, either $T$ is cyclic, or $T = \la x^2,xy\ra$. If $T$ is the latter group, then $[x^2,x^ay] = x^{2(p-1)}\in R$ is reducible, which implies that $x^{2(p-1)}\in\la x^{p+1}\ra$ and so $p+1\mid 2(p-1)$. It follows that $p = 3$, in which case the lemma can be checked easily using {\sc Magma}.
		
		Finally, we assume neither $(d,p) = (6,2)$, nor $d = 2$ and $p = 2^k-1$ for some $k$. By Zsigmondy's theorem \cite{Z_ppd}, it follows that $p^d-1$ has a primitive prime divisor $r$, and we have $x^{(p^d-1)/r}\in T$ since $T$ is transitive (so $r$ divides $|T|$) and $r\nmid d$. Suppose $x^ay^b\in R$ for some $0\leqs a < p^d-1$ and $0\leqs b<d$. Then
		\begin{equation*}
		[x^{(p^d-1)/r},x^ay^b] = x^{(p^d-1)(p^b-1)/r}\in R.
		\end{equation*}
		Since $r\nmid p^b-1$, the order of $x^{(p^d-1)(p^b-1)/r}$ is divisible by $r$ and thus $R$ is irreducible, a contradiction.
	\end{proof}

	\section{The proof of Theorem \ref{thm:reduction}}
	
	\label{s:reduction}
	
	In this section, we will use Theorem \ref{thm:semi} to establish our main result Theorem \ref{thm:reduction} (we will independently prove Theorem \ref{thm:semi} later in Section \ref{s:semi_proof}). As discussed in Section \ref{ss:pre_struc}, we assume Hypothesis \ref{hypo:imprim} throughout, where $K = G_{(\calB)} = G\cap Y^n\ne 1$ and $Y$ is affine (and recall that $Y = G_B^B$). We write $M = \soc(Y)$, so $M\cong \bbZ_p^d$ for some prime $p$ and integer $d$, and thus $|B| = p^d$. Let $L = G\cap M^n$ and $C = C_G(L)$. Note that $L\leqs K$ is an elementary abelian normal $p$-subgroup of $G$.
	
	A key ingredient in the proof of Theorem \ref{thm:reduction} is the following proposition. Recall that $G$ is said to be \textit{semiprimitive} if every normal subgroup of $G$ is either transitive or semiregular. In particular, $G$ is semiprimitive if $G$ is innately transitive.
	
	\begin{prop}
		\label{p:red}
		Assume Hypothesis \ref{hypo:imprim}, where $G_{(\calB)}\ne 1$ and $G_B^B$ is affine. If $G$ has rank $3$ on $\Omega$, then one of the following holds:
		\begin{itemize}\addtolength{\itemsep}{0.2\baselineskip}
			\item[{\rm (i)}] $C = L$;
			\item[{\rm (ii)}] $G^\calB$ is affine and $G$ has a regular normal subgroup;
			\item[{\rm (iii)}] $G^\calB$ is almost simple and $G$ is semiprimitive.
		\end{itemize}
	\end{prop}
	
	\begin{proof}
		Since $L$ is abelian, $L \leqs C$. Thus, we may assume $L<C$ from now on. Note that $C_{K^B}(L^B)\leqs C_{Y}(M) = M$, which yields
		\begin{equation*}
		C\cap K = C_K(L)\leqs G\cap M^n = L.
		\end{equation*}
		Hence, $C\cap K = L$ as $L\leqs K$ and $L\leqs C$. In particular, $C$ is not a subgroup of $K$ (otherwise $C = C\cap K = L$, a contradiction). By Lemma \ref{l:norm_trans}, it follows that $C$ is transitive on $\Omega$ as $C\normeq G$, whence $L\leqs Z(C)$ is semiregular. By Lemma \ref{l:DGLPP:l:2.3}, the set of $L$-orbits on $\Omega$ is $\calB$, and thus $L$ acts regularly on $B$, so $|L| = p^d$. Notice that
		\begin{equation}\label{e:C/L}
		1\ne C/L = C/(K\cap C)\cong CK/K\normeq G/K = G^\calB,
		\end{equation}
		and we may view $C/L$ as a non-trivial normal subgroup of $G^\calB$.
		
		First assume $G^\calB = N{:}T$ is affine, where $N = \soc(G^\calB)\cong \bbZ_r^m$ for some prime $r$.	Observe that $N\normeq C/L$ as $N$ is the unique minimal normal subgroup of $G^\calB$. Thus, $N\leqs O_r(C/L)$, so $Z(O_r(C/L))\leqs C_{G^\calB}(N) = N$. Since $Z(O_r(C/L))\ne 1$ is a characteristic subgroup of $O_r(C/L)$, we have $Z(O_r(C/L))\normeq G^\calB$ and hence $Z(O_r(C/L)) = N$. Therefore, $O_r(C/L)\leqs C_{G^\calB}(N) = N$, which yields $O_r(C/L) = N$ and $C$ has a regular normal subgroup $L.N$, noting that $|L| = |B|$ and $|N| = |\calB|$. If $p = r$, then $O_p(C) = L.N$, which is a regular normal subgroup of $G$ since $O_p(C)$ is a characteristic subgroup of $C\normeq G$. If $p\ne r$, then the extension $L.N$ is split, and moreover, we have $L.N \cong L\times N$ as $L.N\leqs C = C_G(L)$. In this case, $L.N = L\times O_r(C)$ is a regular normal subgroup of $G$.
		
		To complete the proof, we turn to the case where $G^\calB = G/K$ is almost simple. In view of \eqref{e:C/L},  $C^{\infty}\normeq G$ is quasisimple with centre $C^{\infty}\cap K$ and a non-abelian simple composition factor isomorphic to $\soc(G^\calB)$. By Lemma \ref{l:perfect}, we have $[C^{\infty},K] = 1$, and Lemma \ref{l:norm_trans} implies that $C^{\infty}$ is transitive on $\Omega$. It follows that $K\leqs C_G(C^{\infty})$ is semiregular, which implies that $G$ is semiprimitive by Lemma \ref{l:norm_trans}.
	\end{proof}

	We will determine the groups arising in Proposition \ref{p:red}(iii) in Theorem \ref{thm:semi}. To complete the proof of Theorem \ref{thm:reduction} (assuming Theorem \ref{thm:semi}), we may assume that $C = C_G(L) = L$. In view of Corollary \ref{c:KB_trans}, we only need to consider the groups such that $K_{(B)}$ is intransitive on $B'\in\calB\setminus\{B\}$. Now we turn to the case where $K_{(B)} = 1$ and $C = L$. In this setting, we note that $
	L_{(B)} = 1$ and so $L \cong L^B = M\cong\mathbb{Z}_p^d$ acts regularly on $B$, and we will show that there is no rank $3$ group arises.

	\begin{lem}
		\label{l:L_semireg}
		Assume Hypothesis \ref{hypo:imprim}, where $G_{(\calB)}\ne 1$ and $G_B^B$ is affine. If $G$ has rank $3$ on $\Omega$, $C = L$ and $K_{(B)} = 1$, then $G_{(B)} = 1$. Moreover, $G_B/L$ and $G/L$ are subgroups of $\Aut(L)\cong \GL_{d}(p)$ acting faithfully and transitively on the non-identity elements of $L$.
	\end{lem}

	\begin{proof}
		First note that $L\normeq G$ and $C_G(L) = L$, so $G/L\leqs \Aut(L)\cong\GL_d(p)$. Observe that $G_{(B)}\normeq G_B$ and $L\normeq G_B$. By our assumption, we have
		\begin{equation*}
		G_{(B)}\cap L = L_{(B)}\leqs K_{(B)} = 1.
		\end{equation*}
		Hence, $G_{(B)}\leqs C_G(L) = L$, and thus $G_{(B)}  = L_{(B)} = K_{(B)} = 1$. Note that
		\begin{equation*}
		G_B \cong G_B/G_{(B)}=G_B^B
		\end{equation*}
		is an affine $2$-transitive group with socle $L\cong M\cong\bbZ_p^d$. Thus,  $G_B/L\leqs G/L\leqs\Aut(L)$ acts faithfully and transitively on the non-identity elements of $L$.
	\end{proof}
	
	\begin{lem}
		\label{l:C=L,KB=1,Xas}
		Assume Hypothesis \ref{hypo:imprim}, where $G_{(\calB)}\ne 1$ and $G_B^B$ is affine. If $C = L$, $K_{(B)} = 1$ and $G^\calB$ is almost simple, then $G$ is not of rank $3$ on $\Omega$.
	\end{lem}
	
	\begin{proof}
		Suppose $G$ has rank $3$ on $\Omega$ and let $R$ be the socle of $G^\calB = G/K$. Then $R$ is a non-abelian simple composition factor of $G/L$. By Lemma \ref{l:L_semireg}, the groups $G_B/L$ and $G/L$ are both transitive on the non-zero vectors of $\mathbb{F}_p^d$. Note that $|M|^2 = |L|^2 = p^{2d}$ divides $|G_{B,B'}|$ by Lemma \ref{l:DGLPP:l:4.2}, which implies that $p^d$ divides both $|G_B/L|$ and $|G/L|$. Hence, $G_B/L$ and $G/L$ are groups recorded in Corollary \ref{c:trans_linear_prime-power}.
		
		Now we consider each case in Corollary \ref{c:trans_linear_prime-power} in turn. First assume $\SL_a(p^f)\normeq G/L\leqs \mathrm{\Gamma L}_a(p^f)$ for some integer $a\geqs 3$ and prime $p$, where $d = af$. Then $R$ is isomorphic to $\PSL_a(p^f)$, and by inspecting Table \ref{tab:2-trans_as}, we have either $p^f = 2$, $n = 8$ and $a\in\{3,4\}$, or $n = (p^{af}-1)/(p^f-1)$. If the former case holds, then $a\in\{3,4\}$ and $G/L = \GL_a(2)$, which implies that $G_B/L$ is isomorphic to $\bbZ_7{:}\bbZ_3$ or $A_7$. In particular, $G_B/L$ is not one of the groups in Corollary \ref{c:trans_linear_prime-power}, a contradiction. If $n = (p^{af}-1)/(p^f-1)$, then $G_B/L$ has a composition factor $\PSL_{a-1}(p^f)$, and it is easy to see that $G_B/L$ is not a transitive subgroup of $\GL_d(p)$, which is incompatible with Lemma \ref{l:L_semireg}.
		
		Next, assume $\Sp_{2a}(p^f)'\normeq G/L$ for some $a\geqs 2$ and prime $p$, where $d = 2af$, noting that $R$ is isomorphic to $\PSp_{2a}(p^f)'$. By inspecting Table \ref{tab:2-trans_as}, we have $R\cong\Sp_{2a}(2)'$ and $\Omega_{2a}^\pm(2)$ is a composition factor of $G_B/L$. Once again, this is incompatible with Lemma \ref{l:L_semireg}.
		
		Finally, if $G_2(2^f)'\normeq G/L$ and $(p,d) = (2,6f)$, then $f = 1$ and $R\cong G_2(2)'\cong \PSU_3(3)$ by inspecting Table \ref{tab:2-trans_as}. However, this implies that $G_B^\calB$ is soluble, which yields $G_B/L$ is soluble, a contradiction to Lemma \ref{l:L_semireg}.
	\end{proof}

	\begin{lem}
		\label{l:C=L,KB=1,Xaffine}
		Assume Hypothesis \ref{hypo:imprim}, where $G_{(\calB)}\ne 1$ and $G_B^B$ is affine. If $C = L$, $K_{(B)} = 1$ and $G^\calB$ is affine, then $G$ is not of rank $3$ on $\Omega$.
	\end{lem}
	
	\begin{proof}
		Suppose $G$ has rank $3$ on $\Omega$. Write $N = \soc(G^\calB)\cong \bbZ_r^k$ for some prime $r$ and integer $k$. Let $\widetilde{N} = K.N$ be the pre-image of $N$ in $G$ and let $\overline{N} = \widetilde{N}/L$. Then the groups $\overline{K}:=K/L$ and $\overline{N}$ are normal subgroups of $G/L$, with $\overline{N}/\overline{K}\cong N\cong \bbZ_r^k$ an elementary abelian $r$-group. In particular, we have
		\begin{equation*}
		L\normeq K\normeq \widetilde{N}\normeq G\mbox{ and } \overline{K}\normeq \overline{N}\normeq G/L.
		\end{equation*}
		By arguing as in the proof of Lemma \ref{l:C=L,KB=1,Xas}, $G/L$ and $G_B/L$ are one of the transitive linear groups of $\GL_d(p)$ recorded in Corollary \ref{c:trans_linear_prime-power}. In particular, $G$ is insoluble and $(G/L)^{\infty}$ is quasisimple. This allows us to apply Lemma \ref{l:norm_series}, so either $k = 1$ or $G^\calB = G/K$ is metacyclic. In either case, it is easy to see that $N\cong\bbZ_r$ and $G^\calB\cong\mathrm{AGL}_1(r)$, and so $G_{B,B'} = K$ is an insoluble normal subgroup of $G_B$. This implies that $G_{B,B'}/L$ is an insoluble normal subgroup of $G_B/L$. Hence, $G_{B,B'}/L$ is also a transitive linear group recorded in Corollary \ref{c:trans_linear_prime-power}. By Lemma \ref{l:DGLPP:l:4.2}, $G_{B,B'}/L$ has a subgroup of index $p^d$, so $(d,p) = (3,2)$ and $G_{B,B'}/L = \GL_3(2)$ in view of Lemma \ref{l:trans_linear_prime-power_index}. However,
		\begin{equation*}
		\GL_3(2) = G_{B,B'}/L < G/L \leqs \GL_3(2)
		\end{equation*}
		is impossible. This completes the proof.
	\end{proof}

	Now we are in a position to prove Theorem \ref{thm:reduction} assuming the classification in Theorem \ref{thm:semi}.
	
	\begin{proof}[Proof of Theorem \ref{thm:reduction} assuming Theorem \ref{thm:semi}]
		As noted above, we may assume Hypothesis \ref{hypo:imprim}. If $K_{(B)}$ is transitive on $B'$, where $B'\in\calB\setminus\{B\}$, then $G$ has rank $3$ on $\Omega$ by Corollary \ref{c:trans_linear_prime-power}, and $G$ is recorded as in part (C). To complete the proof, we may assume $K\ne 1$ and $K_{(B)}$ is intransitive on $B'$. Note that cases (ii) and (iii) in the statement of Proposition \ref{p:red} are recorded as parts (B) and (A) of Theorem \ref{thm:reduction}, respectively, and the groups in part (A) are determined by combining Theorem \ref{thm:semi} with Theorem \ref{t:innately}. Hence, we may assume $C = L$, so Lemmas \ref{l:C=L,KB=1,Xas} and \ref{l:C=L,KB=1,Xaffine} imply that $K_{(B)}\ne 1$, which gives part (D) of Theorem \ref{thm:reduction}.
	\end{proof}

	\section{Semiprimitive groups}
	
	\label{s:semi_prop}
	
	Recall that a permutation group $G$ is called \textit{semiprimitive} if every normal subgroup of $G$ is either transitive or semiregular. Following \cite[Definition 1.7]{MPR_semi}, a subgroup of $G$ is called an \textit{antiplinth} if it is maximal by inclusion among all intransitive normal subgroups of $G$. It is worth noting that if $G$ is semiprimitive, then $G$ has a unique antiplinth, which is the subgroup generated by all the intransitive normal subgroups of $G$. If $G$ is a rank $3$ semiprimitive permutation group under Hypothesis \ref{hypo:imprim}, then it is easy to see that the antiplinth of $G$ is $K = G_{(\calB)}$, acting regularly on $B$.
	
	Evidently, innately transitive groups are semiprimitive, and in view of Theorem \ref{t:innately}, we will consider the semiprimitive groups that are not innately transitive.
	
	\begin{lem}
		\label{l:semi}
		Assume Hypothesis \ref{hypo:imprim}, where $G$ is semiprimitive but not innately transitive, and $G$ has rank $3$ on $\Omega$. Let $B'\in\calB\setminus\{B\}$, $\beta\in B$ and $\beta'\in B'$.
		\begin{itemize}\addtolength{\itemsep}{0.2\baselineskip}
			\item[{\rm (i)}] $K$ is the unique minimal normal subgroup of $G$, and $K$ is the unique non-trivial intransitive normal subgroup of $G$.
			\item[{\rm (ii)}] Either $K\leqs Z(G)$ or $Z(G) = 1$.
			\item[{\rm (iii)}] The quotient group $(K{:}G_{\beta,\beta'})/K$ is a subgroup of $G_{B,B'}^\calB$ of index $|K|$.
			\item[{\rm (iv)}] $G^\calB_B$ is isomorphic to $G_\beta$.
			\item[{\rm (v)}] $G_B^B$ is affine and $\soc(G_B^B)\cong K$.
		\end{itemize}
	\end{lem}
	
	\begin{proof}
		Let $M$ be a minimal normal subgroup of $G$. Then $M$ is intransitive as $G$ is not innately transitive. This implies that $M\leqs K$ by Lemma \ref{l:norm_trans}. Note that the set of $M$-orbits on $\Omega$ is a non-trivial block system of $G$, which is no other than $\calB$ by Lemma \ref{l:DGLPP:l:2.3}. Hence, we have $M = K$ as $K$ is regular on each block, and so $K$ is the unique minimal normal subgroup of $G$. Let $N\ne 1$ be an intransitive normal subgroup of $G$. Then $N\leqs K$ by Lemma~\ref{l:norm_trans}. Thus, we have $N = K$ as $K$ is minimal normal, and this gives (i). If $Z(G)\ne 1$, then $Z(G)$ is a non-trivial normal subgroup of $G$, so $K\leqs Z(G)$, which implies (ii).
		
		Since $K\cap G_\beta = 1$, we have $G_B = KG_\beta = K{:}G_\beta$ is a semidirect product. By Lemma \ref{l:DGLPP:l:4.2}, this implies that $(K{:}G_{\beta,\beta'})/K\cong G_{\beta,\beta'}$ has index $|K| = |B|$ in $G_{B,B'}/K$. We obtain (iii) by noting that $K = G_{(\calB)}$. Moreover, we have $G_B^\calB = G_B/K\cong G_\beta$, which gives (iv).
		
		Finally, $K_{(B)} = 1$ implies that the image $\overline{K}$ of $K$ in $G_B^B = G_B/G_{(B)}$ is regular on $B$ and is isomorphic to $K$. Thus, $G_B^B$ is affine with socle $\overline{K}\cong K$.
	\end{proof}
	
	For the remainder of this section, we give an infinite family of rank $3$ insoluble semiprimitive groups, which extends \cite[Construction 4.2]{BDP_innately}. Let $X\in\{\GammaL, \GL,\mathrm{\Sigma L},\SL\}$. For convenience, we will simply write $X$ for the insoluble classical group $X_d(q)$ if $d$ and $q = p^f$ are clear from the context, where $p$ is a prime. Consider the natural action of $X$ on $\mathbb{F}_q^d\setminus\{0\}$. Note that $Z(\GL) \cong\bbZ_{q-1}$, which is the subgroup of scalar matrices. Let $C\cong \bbZ_{(q-1)/r}$ be a subgroup of $Z(\GL)$ and let $\Delta$ be the set of $C$-orbits on $\mathbb{F}_q^d\setminus\{0\}$. Then we may write
	\begin{equation}
	\label{e:Delta}
	\Delta=\{\overline{v}:v\in\mathbb{F}_q^d\setminus\{0\}\},
	\end{equation}
	where $\overline{v} =\la \lambda^{r}\ra v$ and $\lambda$ is a generator of $\mathbb{F}_q^\times$, with the action
	\begin{equation*}
	\overline{v}^{gC} = \overline{v^g}
	\end{equation*}
	for any $v\in \mathbb{F}_q^d\setminus\{0\}$ and $g\in \GammaL$. As a special case, if $r = 1$ then $\Delta$ is the set of $1$-subspaces of $\mathbb{F}_q^d$. Here we remark that \cite[Construction 4.2]{BDP_innately} only treats the case where $r$ is a prime divisor of $(q-1)/(d,q-1)$, and we will give more general arguments in this section.
	
	Let us start with the group $\GL$, and the following lemmas are basic exercises of linear algebra.
	
	\begin{lem}
		\label{l:GL_suborb}
		Let $v\in\mathbb{F}_q^d$ be a non-zero vector. Then the orbits of $\GL_v$ are
		\begin{equation}\label{e:GL_suborb}
		\{\lambda v\},\dots,\{\lambda^{q-1}v\},\mathbb{F}_q^d\setminus\la v\ra,
		\end{equation}
		where $\lambda$ is a generator of $\mathbb{F}_q^\times$.
	\end{lem}

	\begin{lem}
		\label{l:m.PGL_suborb}
		Let $\overline{v}\in\Delta$. Then the orbits of $\overline{\GL}_{\overline{v}}$ are
		\begin{equation}\label{e:m.GL_suborb}
		\{\lambda \overline{v}\},\dots,\{\lambda^{r}\overline{v}\}, \{\overline{w}:w\notin \la v\ra\},
		\end{equation}
		where $\lambda$ is a generator of $\mathbb{F}_q^\times$.
	\end{lem}
	
	In particular, $\overline{\GL}$ has rank $3$ if and only if $r = 2$. Now let $\{v_1,\dots,v_d\}$ be a basis of $\mathbb{F}_q^d$ with $v = v_1$, and we treat the group $\GammaL = \la \GL,\phi\ra$. Here $\phi$ is a field automorphism of order $f$ such that
	\begin{equation*}
	(a_1v_1+\cdots +a_dv_d)^\phi = a_1^pv_1+\cdots +a_d^pv_d
	\end{equation*}
	for any $a_1,\dots,a_d\in\mathbb{F}_q$. Let $\overline{\phi}$ be the image of $\phi\in \GammaL$ in $\overline{\GammaL}$. Note that $|\overline{\phi}| = |\phi| = f$.
	
	\begin{lem}
		\label{l:phi_orb}
		We have $(\lambda^\ell \overline{v})^{\overline{\phi}^t} = \lambda^k\overline{v}$ if and only if $k\equiv \ell p^t+sr\pmod{q-1}$ for some $s$.
	\end{lem}
	
	\begin{proof}
		Note that
		\begin{equation*}
		(\lambda^\ell\overline{v})^{\overline{\phi}^t} = \lambda^{\ell p^t}\overline{v} = \lambda^{\ell p^t}\la \lambda^r\ra v.
		\end{equation*}
		Hence, $(\lambda^\ell \overline{v})^{\overline{\phi}^t} = \lambda^k\overline{v}$ if and only if $\lambda^k v = \lambda^{\ell p^t+sr}v$ for some $s$, as desired.
	\end{proof}

	\begin{prop}
		\label{p:m.PGammaL_suborb}
		The group $\overline{\GammaL}$ has rank $3$ on $\Delta$ if and only if $r$ is a primitive prime divisor of $p^{r-1}-1$.
	\end{prop}

	\begin{proof}
		First assume $\overline{\GammaL}$ is of rank $3$. By Lemmas \ref{l:m.PGL_suborb} and \ref{l:phi_orb}, it is easy to see that $\{\lambda^r \overline{v}\}$ and $\{\overline{w}:w\notin\la v\ra\}$ are two $\overline{\GammaL}_{\overline{v}}$-orbits. In view of Lemma \ref{l:phi_orb}, for any $\ell\in\{1,\dots,r-1\}$, there exist integers $s$ and $t$ such that
		\begin{equation}\label{e:mod}
		1\equiv \ell p^t+sr\pmod{q-1}.
		\end{equation}
		In particular, as $r$ divides $q-1$, we have $(\ell,r) = 1$ for every $\ell\in\{1,\dots,r-1\}$, which implies that $r$ is prime. Note that $p^{r-1}\equiv 1\pmod r$ and \eqref{e:mod} implies that
		\begin{equation*}
		1\equiv \ell p^t\pmod r.
		\end{equation*}
		That is, for any $\ell\in\{1,\dots,r-1\}$, there exists a unique $t\in\{0,\dots,r-2\}$ such that $1\equiv \ell p^t\pmod r$. Hence, $1\not \equiv p^t\pmod r$ for any $t\in\{0,\dots,r-2\}$, and so $r$ is a primitive prime divisor of $p^{r-1}-1$, as desired.
		
		Conversely, let $\ell\in\{1,\dots,r-1\}$ and assume that $r$ is a primitive prime divisor of $p^{r-1}-1$. Then there exists $t\in\{0,\dots,r-2\}$ such that $1\equiv \ell p^t\pmod r$. As $r$ divides $q-1$, there exists an integer $s$ such that \eqref{e:mod} holds. By Lemma \ref{l:phi_orb},  $\lambda\overline{v}$ and $\lambda^\ell\overline{v}$ are in a same $\overline{\GammaL}_{\overline{v}}$-orbit. Therefore, the orbits of $\overline{\GammaL}_{\overline{v}}$ are
		\begin{equation*}
		\{\overline{v}\},\{\lambda\overline{v},\dots,\lambda^{r-1}\overline{v}\},\{\overline{w}:w\notin\la v\ra\}.
		\end{equation*}
		This completes the proof.
	\end{proof}

	With this in mind, we may assume that $r$ is a primitive prime divisor of $p^{r-1}-1$ from now on (in particular, we have $r-1\mid f$). We turn to the group $\SL$.
	
	\begin{lem}
		\label{l:SL_suborb}
		Let $v\in\mathbb{F}_q^d$ be a non-zero vector. If $d>2$, then $\SL_v$ and $\GL_v$ have the same set of orbits as presented in \eqref{e:GL_suborb}. If $d = 2$, then the orbits of $\SL_v$ are
		\begin{equation*}
		\{\lambda v\},\dots,\{\lambda^{q-1}v\}, \lambda w+\la v\ra,\dots,\lambda^{q-2}w+\la v\ra,
		\end{equation*}
		where $w\in\mathbb{F}_q^2\setminus\la v\ra$ and $\lambda$ is a generator of $\mathbb{F}_q^\times$.
	\end{lem}
	
	\begin{proof}
		First assume $d>2$ and let $u,w\in \mathbb{F}_q^d\setminus\la v\ra$. If $w = \mu_1v+\mu_2 u$ for some $\mu_1,\mu_2\in\mathbb{F}_q^\times$ then we extend $\{v,u\}$ to a basis $\{v,u,v_3,\dots,v_d\}$ of $\mathbb{F}_q^d$, noting that there exists $x\in \SL_v$ such that $u^x = w = \mu_1v+\mu_2u$, $v_3^x = \mu_2^{-1}v_3$ and $v_i^x = v_i$ for $i\geqs 4$. Similarly, if $\{v,u,w\}$ is linearly independent, then we extend it to a basis $\{v,u,w,v_4,\dots,v_d\}$ of $\mathbb{F}_q^d$. In this setting, there exists $x\in \SL_v$ such that $u^x = w$, $w^x = -u$, $v_i^x = v_i$ for $i\geqs 4$. Thus, $\mathbb{F}_q^d\setminus\la v\ra$ is an orbit of $\SL_v$. Note that $\{\lambda v\},\dots,\{\lambda^{q-1} v\}$ are clearly $\SL_v$-orbits.
		
		To complete the proof, we assume $d = 2$. Let $x\in \SL_v$ and $w\in\mathbb{F}_q^2\setminus\la v\ra$. Once again, it is clear that $x$ fixes $\langle v\rangle$ pointwise. Since $x\in\SL$, we have $w^x = w+\mu v$ for some $\mu\in\mathbb{F}_q^2$. It follows that if $u = a_1v+a_2w$ for some  $a_2\ne 0$, then
		\begin{equation*}
		u^x = (a_1v+a_2w)^x = a_1v+a_2w+a_2\mu v = u+a_2\mu v.
		\end{equation*}
		Thus, $u$ and $u^x$ are in the same affine subset of $\la v\ra$. It is easy to show that any pair of elements in a same non-trivial affine subset are in the same orbit of $\SL_v$.
	\end{proof}

	\begin{lem}
		\label{l:semi_innately}
		Suppose $\overline{\SL}\leqs G\leqs \overline{\GammaL}$ and $G$ has rank $3$ on $\Delta$. Then $G$ is semiprimitive. Moreover, $G$ is not innately transitive if and only if $r\mid (d,q-1)$ and $r\nmid \frac{q-1}{(d,q-1)}$.
	\end{lem}
	
	\begin{proof}
		Since $G$ has rank $3$ on $\Delta$, it satisfies Hypothesis \ref{hypo:imprim} (with $\Delta = \Omega$) and $G$ has a unique non-trivial block system $\calB$, which is induced by $1$-subspaces of $\mathbb{F}_q^d$. It follows that $K = G_{(\calB)} = Z(\overline{\GL})\cong \bbZ_r$, which is semiregular on $\Delta$.
		Thus, $G$ is semiprimitive by Lemma \ref{l:norm_trans}.
		
		Suppose $G$ is not innately transitive on $\Delta$. Note that $\overline{\SL}\normeq G$ is transitive on $\Delta$, so $Z(\overline{\GL}) = K  < \overline{\SL}$ (otherwise $\overline{\SL}$ is a transitive minimal normal subgroup) and thus $Z(\overline{\SL}) = Z(\overline{\GL})$. This implies that $r\mid (d,q-1)$. Write $Z(\GL) = \la x\ra$, so $C = \la x^r\ra$ and $Z(\SL) = \la x^{\frac{q-1}{(d,q-1)}}\ra$. If $r\mid \frac{q-1}{(d,q-1)}$, then
		\begin{equation*}
		\overline{\SL} = \SL/(C\cap \SL) = \SL/(C\cap Z(\SL)) = \SL/Z(\SL)\cong \PSL_d(q),
		\end{equation*}
		which is incompatible with $Z(\overline{\GL})<\overline{\SL}$. Hence, $r\nmid \frac{q-1}{(d,q-1)}$.
		
		Finally, suppose $G$ is innately transitive and let $N$ be a transitive minimal normal subgroup of $G$. Then $K\cap N = 1$ and $N/(K\cap N)\cong N$ is a normal subgroup of $G/K = G^\calB$. Note that $\soc(G^\calB) = \PSL_d(q)$, so $N\cong \PSL_d(q)$. This yields $\PSL_d(q)\cong N= \overline{\SL}$. By arguing as above, we have $Z(\SL)\leqs C$ and so $r\mid \frac{q-1}{(d,q-1)}$.
	\end{proof}

	Now consider the suborbits of $\overline{\SL} = \SL/(C\cap \SL)\leqs \mathrm{Sym}(\Delta)$. Here we exclude the cases where $d = 2$ and $r\ne2$ (as noted in Lemma \ref{l:semi_innately}, if $G$ has rank $3$ on $\Delta$ then $G$ is innately transitive in this setting, and we refer the reader to \cite[Proposition 4.3]{BDP_innately} for details).
	
	\begin{lem}
		\label{l:m.PSL_suborb}
		Let $\overline{v}\in\Delta$. If $d \geqs 3$, then $\overline{\SL}_{\overline{v}}$ and $\overline{\GL}_{\overline{v}}$ have the same set of orbits as presented in \eqref{e:m.GL_suborb}. If $d = 2$, $r =2$ and $p$ is odd, then the orbits of $\overline{\SL}_{\overline{v}}$ are
		\begin{equation}\label{e:m.PSL2_orb}
		\{\lambda\overline{v}\},\{\lambda^2\overline{v}\}, \{\overline{\lambda^kw+\mu v}:\mu\in\mathbb{F}_q\mbox{, $k$ odd}\}, \{\overline{\lambda^kw+\mu v}:\mu\in\mathbb{F}_q\mbox{, $k$ even, $q-1\nmid k$}\},
		\end{equation}
		where $w\in\mathbb{F}_q^2\setminus\la v\ra$ and $\lambda$ is a generator of $\mathbb{F}_q^\times$.
	\end{lem}

	In view of Lemmas \ref{l:SL_suborb} and \ref{l:m.PSL_suborb}, we will consider the cases where $d>2$ and $(d,r) = (2,2)$ separately.
	
	\begin{prop}
		\label{p:n>2_rank3}
		Suppose $\overline{\SL}\leqs G\leqs\overline{\GammaL}$ and $d>2$. Then $G$ has rank $3$ on $\Delta$ if and only if the following conditions hold:
		\begin{itemize}\addtolength{\itemsep}{0.2\baselineskip}
			\item[{\rm (i)}] $r$ is a primitive prime divisor of $p^{r-1}-1$; and
			\item[{\rm (ii)}] $(r-1,\frac{f|G\cap\overline{\GL}|}{|G|}) = 1$.
		\end{itemize}
	\end{prop}
	
	\begin{proof}
		As noted in Proposition \ref{p:m.PGammaL_suborb}, we may assume $r$ is a primitive prime divisor of $p^{r-1}-1$, and it suffices to show that $G$ is of rank $3$ if and only if the condition (ii) holds. By Lemma \ref{l:m.PSL_suborb}, the orbits of $\overline{\SL}_{\overline{v}}$ and $\overline{\GL}_{\overline{v}}$ coincide, which implies that $(\lambda^\ell\overline{v})^{\overline{\delta}} = \lambda^\ell\overline{v}$ for any $\ell\in\{1,\dots,r\}$ (here $|\delta| = q-1$ with $\GL = \la\SL,\delta\ra$, and $\overline{\delta}$ is its image in $\overline{\GammaL}$). It is worth noting that the condition (ii) is equivalent to
		\begin{equation}
		\label{e:star}
		\begin{aligned}
		&\mbox{for any $t\in\{0,\dots,r-2\}$, there exists $j$ such that}
		\\
		&\mbox{$j\equiv t\pmod{r-1}$ and $\overline{\delta^i\phi^j}\in G$ for some $i$.}
		\end{aligned}
		\tag{$\star$}
		\end{equation}
		
		Suppose \eqref{e:star} holds, noting that $p^t\equiv p^j\pmod{r}$ and $1,p,\dots,p^{r-2}$ are distinct modulo $r$. This implies that for any $\ell\in\{1,\dots,r-1\}$, there exists an integer $j$ such that $1\equiv \ell p^{j}\pmod r$ and $\overline{\delta^i\phi^j}\in G$ for some $i$. Since $r\mid q-1$, we have $1\equiv \ell p^{j}+sr\pmod {q-1}$ for some integer $s$. By Lemma \ref{l:phi_orb},
		\begin{equation*}
		(\lambda^\ell\overline{v})^{\overline{\delta^i\phi^j}} = (\lambda^\ell\overline{v})^{\overline{\phi^j}} = \lambda \overline{v}.
		\end{equation*}
		Hence, $\{\lambda\overline{v},\dots,\lambda^{r-1}\overline{v}\}$ is a $G_{\overline{v}}$-orbit, and so $G$ is of rank $3$.
		
		To complete the proof, we assume $G$ is of rank $3$. This implies that the orbits of $G_{\overline{v}}$ are
		\begin{equation*}
		\{\overline{v}\},\{\lambda\overline{v},\dots,\lambda^{r-1}\overline{v}\},\{\overline{w}:w\notin\la v\ra\}.
		\end{equation*}
		Suppose \eqref{e:star} fails to hold. That is, there exists $t\in\{0,\dots,r-2\}$ such that for any $\overline{\delta^i\phi^j}\in G$, we have $j\not\equiv t\pmod{r-1}$, which is equivalent to $p^t\not\equiv p^j\pmod{r}$ as $r$ is a primitive prime divisor of $p^{r-1}-1$. Let $\ell\in\{1,\dots,r-1\}$ be such that $1\equiv \ell p^t\pmod{r}$. It follows from Lemma \ref{l:phi_orb} that there exists $\overline{\delta^i\phi^k}\in G$ such that $1\equiv \ell p^k+sr \pmod{q-1}$ for some $s$, since $(\lambda^\ell\overline{v})^{\overline{\delta}} = \lambda^\ell\overline{v}$. However, this implies that $1\equiv \ell p^k\pmod r$ and so $p^t\equiv p^k\pmod r$, which is incompatible with our assumption.
	\end{proof}
	
	Finally, let us consider the case where $(d,r) = (2,2)$ and $p$ is odd. The insolubility implies that $q\geqs 5$. Let $w\in\mathbb{F}_q^2\setminus\la v\ra$ and $\phi$ be the field automorphism of order $f$ such that
	\begin{equation*}
	(a_1v+a_2w)^\phi = a_1^pv+a_2^pw
	\end{equation*}
	for any $a_1,a_2\in\mathbb{F}_q$. Recall that $\SigmaL = \la\SL,\phi\ra$ and $\overline{\SigmaL} =\la\overline{\SL},\overline{\phi}\ra$.
	
	\begin{prop}
		\label{p:(d,m) = (2,2)}
		Suppose $\overline{\SL}\leqs G\leqs\overline{\GammaL}$, $(d,r) = (2,2)$ and $p$ is odd. Then $G$ has rank $3$ on $\Delta$ if and only if $G$ is not a subgroup of $\overline{\SigmaL}$.
	\end{prop}
	
	\begin{proof}
		First note that the field automorphism $\overline{\phi}$ preserves the $\overline{\SL}_{\overline{v}}$-orbits, as presented in \eqref{e:m.PSL2_orb}. Thus, if $G\leqs \overline{\SigmaL} = \overline{\SL}.\la \overline{\phi}\ra$, then $\overline{\SL}_{\overline{v}}$ and $G_{\overline{v}}$ have the same set of orbits, and so $G$ has rank $4$ on $\Delta$ as $q\geqs 5$ by our assumption. Now assume that $G$ is not a subgroup of $\overline{\SigmaL}$, and so $\overline{\delta\phi^j}\in G$ for some $j$. Once again, as $\overline{\phi}$ preserves the $\overline{\SL}_{\overline{v}}$-orbits, the set of $G_{\overline{v}}$-orbits are the set of $\overline{\GL}_{\overline{v}}$-orbits as presented in \eqref{e:m.GL_suborb}. Therefore, $G$ has rank $3$ on $\Delta$.
	\end{proof}
	
	To sum up, we record the rank $3$ groups discussed above as follows.

	\begin{prop}
		\label{p:PSL}
		Suppose $q = p^f$ for some prime $p$ and $C\cong\bbZ_{(q-1)/r}$ is a subgroup of $Z(\GL_d(q))$. Let
		\begin{equation*}
		\overline{\cdot}:\GammaL_d(q)\to \GammaL_d(q)/C
		\end{equation*}
		be the quotient map. Suppose $\overline{\SL_d(q)}\leqs G\leqs\overline{\GammaL_d(q)}$ and $G$ is not innately transitive on $\Omega$, where $\Omega$ is the set of $C$-orbits on $\mathbb{F}_q^d\setminus\{0\}$. Then $G$ has rank $3$ on $\Omega$ if and only if $r$ is a primitive prime divisor of $p^{r-1}-1$, $r\nmid \frac{q-1}{(d,q-1)}$, and one of the following holds:
		\begin{itemize}\addtolength{\itemsep}{0.2\baselineskip}
			\item[{\rm (i)}] $d\geqs 3$ and $(r-1,\frac{f|G\cap\overline{\GL_d(q)}|}{|G|}) = 1$; or
			\item[{\rm (ii)}] $(d,r) = (2,2)$ and $G$ is not a subgroup of $\overline{\SigmaL_2(q)}$.
		\end{itemize}
		In particular, we have $r-1\mid f$, and moreover, $G$ is semiprimitive if $G$ has rank $3$ on $\Omega$.
	\end{prop}

	The final observation in this section will be used in proving Theorem \ref{thm:semi} in Section \ref{s:semi_proof}.

	\begin{lem}
		\label{l:normalsl}
		$N_{\Sym(\Delta)}(\overline{\SL})=\overline{\GammaL}$.
	\end{lem}
	
	\begin{proof}
		Let $N:=N_{\Sym(\Delta)}(\overline{\SL})$, and let $M:=C_{\Sym(\Delta)}(\overline{\SL})$.
		It is clear that $\overline{\GammaL}\leqslant N$ and
		\begin{equation*}
		\overline{\SL}/(M\cap \overline{\SL})=\overline{\SL}/Z(\overline{\SL})\cong\PSL_d(q).
		\end{equation*}
		Let $\calB$ be the set of $M$-orbits on $\Delta$, which forms a block system of $N$ on $\Delta$. Then the permutation image $\overline{\SL}^\calB$ of $\overline{\SL}$ on $\calB$ is isomorphic to $\PSL_d(q)$ or $1$.
		Since $\overline{\SL}$ is not regular on $\Delta$, we deduce that $M$ is intransitive on $\Delta$, which implies that $|\calB|\neq 1$ and so $\overline{\SL}^\calB\cong\PSL_d(q)$.
		Note that $\overline{\SL}$ acts $2$-transitively on the set $\Delta_{Z(\overline{\GL})}$ of orbits of $Z(\overline{\GL})$ on $\Delta$, of which $\calB$ is a block system. It follows that $\Delta_{Z(\overline{\GL})} = \calB$, and thus  $\overline{\GammaL}^\calB\cong\PGammaL_d(q)$ and $M=Z(\overline{\GL})$ as $M$ is semiregular. Now
		\begin{equation*}
		\PGammaL_d(q)\cong \overline{\GammaL}^\calB\leqs N^\calB\leqs N_{\mathrm{Sym}(\calB)}(\overline{\SL}^\calB)\cong \PGammaL_d(q)
		\end{equation*}
		and so $N^\calB = \overline{\GammaL}^\calB$.
		Note that $N_{(\calB)}\cap\overline{\SL}=\overline{\SL}_{(\calB)} = Z(\overline{\SL})$, which implies that $[N_{(\calB)},\overline{\SL}] = 1$ by Lemma \ref{l:perfect}.
		It follows that $N_{(\calB)}\leqs M$, and thus $N_{(\calB)} = M$ since $\calB$ is the set of $M$-orbits on $\Delta$.
		Now $N/M = N^\calB= \overline{\GammaL}^\calB = \overline{\GammaL}/M$, which yields $N=\overline{\GammaL}$.
	\end{proof}

	\section{The proof of Theorem \ref{thm:semi}}
	
	\label{s:semi_proof}

	We will prove Theorem \ref{thm:semi} in this section, so we may assume Hypothesis \ref{hypo:imprim}, where $G$ is semiprimitive but not innately transitive.
	
	\subsection{The groups with $G^\calB$ almost simple}
	
	We first consider the case where $G^\calB= G/K$ is almost simple with socle $R$. Recall that if $G$ has rank $3$, then $K$ is abelian by Lemma \ref{l:semi}(v). And we have $C_G(K)/K\normeq G/K$, so either $K.R\leqs C_G(K)$ or $C_G(K) = K$. We first show that the latter case does not arise.
	
	\begin{lem}
		\label{l:C_G(K)_ne_K}
		Assume Hypothesis \ref{hypo:imprim}, where $G$ is semiprimitive but not innately transitive and $G^\calB$ is almost simple. If $G$ has rank $3$ on $\Omega$, then $K.R\leqs C_G(K)$ and $K.R$ is quasisimple.
	\end{lem}
	
	\begin{proof}
		Suppose that $G$ has rank $3$ on $\Omega$ and $K = C_G(K)$. Since $K\cap G_{(B)} = 1$ and $G_{(B)}\normeq G_B$, we have $K\times G_{(B)}\leqs C_G(K)\cap G_B = K$, which implies that $G_{(B)} = 1$. Let $L = G\cap \soc(Y)^n$, noting that $L\normeq G$ is intransitive. By Lemmas \ref{l:norm_trans} and \ref{l:semi}(i), we have $L = K$, whence
		\begin{equation*}
		C_G(L) = C_G(K) = K = L,
		\end{equation*}
		which is incompatible with Lemma \ref{l:C=L,KB=1,Xas}.

		This shows that $K\ne C_G(K)$, and so $K.R\leqslant C_G(K)$ as noted above. Now let $H=(K.R)^\infty$.
		Then $H\neq 1$ is normal in $G$, and so is $H\cap K$.
		By Lemma~\ref{l:semi}(i), we see that $H\cap K = 1$ or $K$.
		If the former case holds, then $H\cong R$ is a minimal normal subgroup of $G$, and hence $H$ is a transitive by Lemma~\ref{l:semi}(i), which yields $G$ is innately transitive.
		Thus, we have $H\cap K=K$ and so $K\leqs H$. This implies that $H=K.R$ is a perfect group with $Z(H)=K$.
		Therefore, $K.R$ is quasisimple.
	\end{proof}

The following lemma describes the possibilities of the group $K.R$.
	
	\begin{lem}
		\label{l:K.R}
		Assume Hypothesis \ref{hypo:imprim}, where $G$ is semiprimitive but not innately transitive, $H$ is a point stabiliser of $G$, and assume $G^\calB$ is almost simple. If $G$ has rank $3$ on $\Omega$, then one of the following holds (up to permutation isomorphism):
		\begin{itemize}\addtolength{\itemsep}{0.2\baselineskip}
			\item[{\rm (a)}] $3\mid q+1$ and $K.R \cong \SU_3(q)$.
			\item[{\rm (b)}] $R = \PSL_d(q)$ acts on $1$-spaces of $\mathbb{F}_q^d$, and $K.R$ is a quotient group of $\SL_d(q)$.
			\item[{\rm (c)}] $(G,H) \in\{ (3.S_6,S_5), (2.\mathrm{M}_{12},\mathrm{M}_{11})\}$.
		\end{itemize}
	\end{lem}
	
	\begin{proof}
		Lemma~\ref{l:C_G(K)_ne_K} shows that $K.R$ is quasisimple, and so $K$ is a subgroup of the Schur multiplier $m(R)$ of $R$. Note that $K$ is elementary abelian, so by inspecting Table \ref{tab:2-trans_as}, if $K$ is not cyclic, then $K = 2^2$ and $R\in\{{^2}B_2(8),\PSL_3(4)\}$. That is, $G = 2^2.G^\calB$ with $\soc(G^\calB)\in \{{^2}B_2(8),\PSL_3(4)\}$. With the aid of {\sc Magma}, one can check that no rank $3$ semiprimitive group arises in this setting. More specifically, we first use the function \texttt{ExtensionsOfElementaryAbelianGroup} to obtain all the abstract groups of the form $2^2.G^\calB$. Among these groups, it is routine to get the groups with a unique minimal normal subgroup and a subgroup of order $|H| = |G_B^\calB|$. For each of the remaining groups, it is easy to check that its action on the cosets of any subgroup of order $|G_B^\calB|$ is not semiprimitive.
		
		It follows that $K$ is cyclic. Since $G$ is not quasiprimitive, we have $K\ne 1$ is of prime order. Now we consider each case in Table \ref{tab:2-trans_as} in turn. First assume $K.R = 2.A_m$, with $R$ acting naturally on $m$ points. In this setting, either $G = 2.A_m$ or $G = 2.S_m$, and $G$ has degree $2m$. However, \cite[Lemma 2.10]{MPR_semi} shows that $G$ has no permutation representation of degree $2m$, so this case does not arise.
		
		Next, we note that the case where $R = \PSU_3(q)$ acting on isotropic $1$-spaces is recorded as case (a) (the condition $3\mid q+1$ is equivalent to $K\ne 1$). One can also see that if $R = \PSL_d(q)$ with its parabolic action on $1$-spaces and $m(R) = (d,q-1)$, then $K.R$ is a quotient group of $\SL_d(q)$, which is case (b).
		
		This remains finitely many cases listed in Table \ref{tab:2-trans_as} to consider. In some cases, we use the generators given in \cite{W_WebAt} to construct $G$ in {\sc Magma}. By arguing as above, one can check using {\sc Magma} that if $G$ is a rank $3$ group satisfying our conditions, then $(G,H)= (3.S_6,S_5)$ or $(2.\Ma_{12},\Ma_{11})$.	This gives case (c) and completes the proof.
	\end{proof}
	
	It is worth noting that the groups listed in Lemma \ref{l:K.R}(c) are genuine examples of semiprimitive permutation groups of rank $3$, which are not innately transitive. We first consider case (a) in Lemma \ref{l:K.R}, and we will show that this case does not arise.

	\begin{lem}
		\label{l:PSU_non-exist}
		Assume Hypothesis \ref{hypo:imprim}, where $G$ is semiprimitive but not innately transitive and $G^\calB$ is almost simple. If $3\mid q+1$ and $K.R \cong \SU_3(q)$, then $G$ is not of rank $3$ on $\Omega$.
	\end{lem}

	\begin{proof}
		Suppose $G$ has rank $3$ on $\Omega$ and set $S:=G^{\infty}=K.R\cong\SU_3(q)$, noting that $|S:S_B|=q^3+1$, which implies that $S_B\cong [q^3]{:}\bbZ_{q^2-1}$.
		Then it is easy to see that a Sylow $3$-subgroup $P$ of $S_B$ is cyclic, and $K\normeq P$ is regular on $B$. It follows that $P = K\cong\bbZ_3$, and so $(q+1)_3 = 1$. Thus, $q\equiv 2$ or $5\pmod 9$, so $3\nmid f$ by noting that there is no integer $q_0$ such that $q_0^3\equiv 2$ or $5\pmod 9$, where $q = p^f$ for prime $p$. In particular, this implies that $(|G_B|)_3\leqs 2$, as $|\Out(\PSU_3(q))| = 3f$.
		Moreover, Lemma \ref{l:semi}(iii) implies that $G_{B,B'}$ has a subgroup of index $|K|^2=9$, which yields $(|G_B|)_3\geqs 2$.	Thus, we have $(|G_B|)_3 = 2$, and so $\PGU_3(q)\leqs G^\calB$ and a Sylow $3$-subgroup $Q$ of $G_B$ has order $9$.
		Since $Q = Q_B=K{:}Q_\beta$ (where $\beta\in B$) with $K\cong\bbZ_3$, we obtain that $Q\cong\bbZ_3^2$.
		
		Let $\{v_1,v_2,v_3\}$ be a basis of $\mathbb{F}_{q^2}^3$. We may identify the unitary form on $\mathbb{F}_{q^2}^3$ with
		\begin{equation*}
		(a_1v_1+a_2v_2+a_3v_3)(b_1v_1+b_2v_2+b_3v_3) = a_1b_3^q+a_2b_2^q+a_3b_1^q
		\end{equation*}
		for $a_i,b_j\in\mathbb{F}_{q^2}$. We also identify $\calB$ with the set of $1$-isotropic subspaces. Let $B = \la v_1\ra $ and $B' = \la v_3\ra$. Recall that $\PGU_3(q)\leqs G^\calB$ and note that $G_{B,B'}^\calB\cap \PGU_3(q)\cong\mathbb{Z}_{q^2-1}$ comprises the images of diagonal matrices $\mathrm{diag}(a,1,a^{-q})$ for $a\in\mathbb{F}_{q^2}^\times$, which has a Sylow subgroup of order $3$. Thus, a Sylow subgroup of $G_{B,B'}$ has order $9$ as $3\nmid f$. With this in mind, we may assume that $Q\leqs G_{B,B'}$. Now let $x\in\PGU_3(q)$ be the image of
		\begin{equation*}
		\begin{pmatrix}
		0&0&1\\
		0&1&0\\
		1&0&0
		\end{pmatrix}
		\in\GU_3(q),
		\end{equation*}
		so $|x| = 2$. Note that $\mathrm{diag}(a,1,a^{-q})^x = \mathrm{diag}(a,1,a^{-q})^{-q}$, and for any $\lambda\in\mathbb{F}_q^\times$ with $|\lambda| = 3$ we have $\lambda^{-q} = \lambda$ since $3\mid q+1$. It follows that $x$ centralises $Q/K$, which is a Sylow $3$-subgroup of $G_{B,B'}^\calB$.
		
		Since $(|K|,|x|) = 1$, there exists a pre-image $\widehat{x}\in G$ of $x = K\widehat{x}$ such that $|\widehat{x}| = |x| = 2$. As $K = Z(K.R)$ and $|\PGU_3(q):R| = 3$ is coprime to $|K|-1 = 2$, we have $K = Z(K.\PGU_3(q))$ and thus $[\widehat{x},K] = 1$. Moreover, we have $[x,Q/K] = 1$ as noted above, and hence $\widehat{x}$ induces a unipotent element in $\Aut(Q)\cong\GL_2(3)$, which implies that the induced automorphism is trivial as $|\widehat{x}| = 2$. Therefore, we have $[\widehat{x},Q]=1$.
		
		Now let $\beta\in B$. By Lemma \ref{l:semi}(iii), the group $G_\beta\cap G_\beta^{\widehat{x}}$ has index $9$ in $G_{B,B'} = G_B\cap G_B^{\widehat{x}}$, so $3$ does not divide $|G_\b\cap G_\b^{\widehat{x}}|$ since $(|G_B|)_3 = 2$.
		Thus, $Q_\beta\cap Q_\beta^{\widehat{x}}=1$.
		However, since $|Q_\beta|=|Q|/|K|=3$ and $[Q_\beta,\widehat{x}]=1$, we have $Q_\beta\cap Q_\beta^{\widehat{x}}=Q_\beta$, a contradiction.	
		This concludes that $G$ is not of rank $3$ on $\Omega$.
	\end{proof}
	
	Finally, we turn to case (b) recorded in Lemma \ref{l:K.R}. Here $R\cong\PSL_d(q)$, $K\cong\bbZ_r$ for some prime $r\mid(d,q-1)$ and $G^{\infty}=K.R$ is a quotient of $\SL_d(q)$. Recall that $G^\infty \cong \overline{\SL_d(q)}$ with the notation in Proposition \ref{p:PSL}.
	
	\begin{lem}
		\label{l:isoolS}
		If $G$ has rank $3$ on $\Omega$ and satisfies case {\rm (b)} in Lemma \ref{l:K.R}, then $G$ is permutation isomorphic to a group of rank $3$ as described in Proposition \ref{p:PSL}.
	\end{lem}
	
	\begin{proof}
		First note that $G\leqs N_{\Sym(\Omega)}(G^{\infty})$, so in view of Lemma \ref{l:normalsl}, we only need to show that $T:=G^{\infty}$ is permutation isomorphic to $\overline{\SL_d(q)}$ described in Proposition \ref{p:PSL}. In other words, it suffices to show that $T_\beta$ is $\Aut(T)$-conjugate to a point stabiliser of $\overline{\SL_d(q)}$, where $\b\in\Omega$.
		
		Note that $T_B = K{:}T_\beta$ and
		\begin{equation*}
		T_\beta\cong (K{:}T_\beta)/K = T_B^\calB \cong [q^{d-1}]{:}\GL_{d-1}(q)/\bbZ_{(d,q-1)}.
		\end{equation*}
		Thus, $T_\beta$ has a normal $p$-subgroup $P$ of order $q^{d-1}$, and so the pre-image $\widetilde{T_\beta}$ of $T_\beta\leqs\overline{\SL_d(q)}$ in $\SL_d(q)$ also has a normal $p$-subgroup $\widetilde{P}$. Now the centre of $\mathbb{F}_q^d{:}\widetilde{P}$ is fixed by $\widetilde{T_\beta}$, whence $\widetilde{T_\beta}$ stabilises a non-trivial proper subspace $W$ of $\mathbb{F}_q^d$. If $\dim W\notin\{1,d-1\}$, then it is easy to see that $|\SL_d(q)_W| < |\widetilde{T_\beta}|$, a contradiction. It follows that $\dim W = 1$ or $d-1$, and hence $\widetilde{T_\beta}$ is $\Aut(\SL_d(q))$-conjugate to a subgroup of a maximal subgroup $M$ of $\SL_d(q)$ of type $P_1$. By repeating the above argument, one can see that the pre-image of a point stabiliser of $\overline{\SL_d(q)}$ in $\SL_d(q)$ is also $\Aut(\SL_d(q))$-conjugate to a subgroup of $M$.
		
		Therefore, it suffices to show that $M$ has a unique subgroup of index $r = |K|$. Recall that $M\cong [q^{d-1}]{:}\GL_{d-1}(q)$. If $d = 2$, then $M\cong [q]{:}\bbZ_{q-1}$ and it is evident that $M$ has a unique subgroup of index $r = 2$. If $d\geqs 3$, then $(d,q)\ne (3,2)$ or $(3,3)$ as $r\mid (d,q-1)$, and so $\GL_{d-1}(q)$ is insoluble. We claim that $M$ also has a unique subgroup of index $r$. To see this, note that if $M_0$ is a subgroup of $M$ of index $r$, then $[q^{d-1}] \normeq M_0$ as $(r,q) = 1$. Hence, $M_0/[q^{d-1}]$ is a subgroup of $M/[q^{d-1}]\cong \GL_{d-1}(q)$ of index $r$. Now the claim follows, as $\GL_{d-1}(q)$ has a unique subgroup of index $r$. This completes the proof.
	\end{proof}
	
	\begin{prop}
		\label{p:GB_as_semi}
		Assume Hypothesis \ref{hypo:imprim}, where $G$ is semiprimitive but not innately transitive, $H$ is a point stabiliser of $G$, and assume $G^\calB$ is almost simple. If $G$ has rank $3$ on $\Omega$, then either $G$ is a group as described in Proposition \ref{p:PSL}, or $(G,H) = (3.S_6,S_5)$ or $(2.\mathrm{M}_{12},\mathrm{M}_{11})$.
	\end{prop}
	
	\begin{proof}
		Combine Lemmas \ref{l:K.R}, \ref{l:PSU_non-exist} and \ref{l:isoolS}.
	\end{proof}

	\subsection{The groups with $G^\calB$ affine}	
	
	Now we continuously assume Hypothesis \ref{hypo:imprim}, where $G$ is semiprimitive, and we turn to the case where $G^\calB = G/K$ is affine with socle $M\cong \bbZ_r^m$ for some prime $r$. As noted in \cite[Lemma 2.5(a)]{BDP_innately}, $G$ is not innately transitive, and we recall that $K\cong \bbZ_p^k$ for some prime $p$.

	\begin{lem}
		\label{l:p=r_affine}
		Assume Hypothesis~\ref{hypo:imprim}, where $G$ is semiprimitive and $G^\calB$ is affine. 
		If $G$ has rank $3$ on $\Omega$, then $G$ has a regular normal subgroup $N$, where $N$ is a $p$-group.
	\end{lem}
	
	\begin{proof}
		Let $N := K.M$ the preimage of $M$ in $G$. Note that $M$ is transitive on $\calB$ and $K$ is transitive on $B\in\calB$, it follows that $N$ is transitive on $\Omega$.
		Thus, $N$ is a regular normal subgroup of $G$ since $|N|=|K||M|=|B||\calB|=|\Omega|$.
		
		Suppose that $G$ has rank $3$ on $\Omega$ and let $\beta\in\Omega$. Then $G_{\beta}$ is isomorphic to a subgroup of $\Aut(N)$, acting on $N$ with exactly three orbits.
		Hence, $N$ is either a Frobenius $\{p,q\}$-group or a $p$-group by Lemma~\ref{l:regularnormal}.
		We argue by contradiction and suppose that $N$ is a Frobenius $\{p,q\}$-group.
		Note that $K\normeq G$ and $G$ acts on $N=K.M$ with exactly three orbits.
		Hence, $G$ acts transitively on the non-identity elements of $K$, and $M = N/K$ is a $q$-group.
		By Lemma~\ref{l:regularnormal}, we have $N=K{:}Q$ with $M\cong Q\cong\bbZ_q$.
		This implies that $G^{\calB}\cong\bbZ_q{:}\bbZ_{q-1}$ as $M\cong\bbZ_q$ is the socle of $G^\calB$.
		That is, $G = N{:}G_\beta$ is isomorphic to a subgroup of $(\bbZ_p^k{:}\bbZ_q){:}\bbZ_{q-1}$, where $\beta\in\Omega$. 
		However, since $q-1<p^k-1$, it is clear that there are at least $4$ orbits of $G_\beta$ on $N$, which implies that $G$ is not of rank $3$, a contradiction. 
		Therefore, we conclude that $N$ is a $p$-group.
	\end{proof}

	In order to complete the proof of Theorem \ref{thm:semi} and in view of Lemma \ref{l:regularnormal}, it suffices to consider the case where $m = k$ and $N \cong \bbZ_{p^2}^k$. In this setting, we have $\Aut(N)\cong\GL_k(\bbZ/p^2\bbZ)\cong \bbZ_p^{k^2}.\GL_k(p)$, and we first list some preliminary lemmas.
	
	\begin{lem}
		\label{l:S=T}
		Assume Hypothesis \ref{hypo:imprim}, where $G$ is semiprimitive and $G^\calB$ is affine. If  $N\cong\bbZ_{p^2}^k$ and $G$ has rank $3$ on $\Omega$, then $G_B^\calB\cong G_\beta^B$ and $G_B^\calB$ is permutation isomorphic to one of the transitive linear groups listed in Corollary \ref{c:trans_linear_prime-power}.
	\end{lem}
	
	\begin{proof}
		First, by Lemma \ref{l:semi}(iii), we have $p^k$ divides $|G_B^\calB|$, and so $G_B^\calB$ is one of the transitive linear groups listed in the statement of Corollary \ref{c:trans_linear_prime-power}. It suffices to show that $G_B^\calB\cong G_\beta^B$. Since $N$ is a regular normal subgroup of $G$ and $K\leqs N$, we may identify $\Omega = N$,
		\begin{equation*}
		\calB = [N:K] = \{Kx\mid x\in N\}
		\end{equation*}
		and $B = K$. Hence, $G_{(B)}$ fixes all elements of $N$ of order $p$. Note that for any $x\in N$ and $g\in G_{(B)}$, if $(x^g)^p = x^p$, then $x^g = yx$ for some $y\in K$. This implies that $G_{(B)}$ induces a trivial action on $\calB$, whence $G_{(B)}\leqs G_{(\calB)} = K$. As $K$ is semiregular on $N$, we have $G_{(B)} = K_{(B)} = 1$. Now apply Lemma \ref{l:semi}(iv).
	\end{proof}

	\begin{lem}
		\label{l:pre-image_p=2,3}
		Assume Hypothesis \ref{hypo:imprim}, where $G$ is semiprimitive and $G^\calB$ is affine. If $N\cong \bbZ_{p^2}^k$ and $G$ has rank $3$ on $\Omega$ with a composition factor isomorphic to $\PSL_k(p)$ or $\PSp_k(p)$, then $p = 2$ or $3$.
	\end{lem}
	
	\begin{proof}
		Let $\beta\in\Omega$ and suppose $p\ne 2$ or $3$. First note that $G=N{:}G_\beta$ and so $G_\beta\leqs \Aut(N)\cong\GL_k(\bbZ/p^2\bbZ)$ has a composition factor isomorphic to $\PSL_k(p)$ or $\PSp_k(p)$.
		It follows that $(G_B^\calB)^\infty\cong (G_\beta)^{\infty}\cong\SL_k(p)$ or $\Sp_k(p)$ is a subgroup of $\GL_k(\bbZ/p^2\bbZ)$.
		Hence, with respect to a standard basis, there exists an element $A=I+E+p M$ in $G_\beta$, where $I$ is the identity matrix,			
		\begin{equation*}
		E = \begin{pmatrix}
		0&1&0&\cdots& 0\\
		0&0&0&\cdots& 0\\
		\vdots&\vdots&\vdots&\ddots&\vdots\\
		0&0&0&\cdots& 0
		\end{pmatrix}
		\end{equation*}
		and $M$ is a $(k\times k)$-matrix.
		
		Recall that the kernel $R$ of $\GL_k(\bbZ/p^2\bbZ)$ acting on $\Phi(N)$ consists of matrices of the form $I+pM$ for some $(k\times k)$-matrix $M$.
		Notice that $\Phi(N)$ and $\Phi(N){:}(R\cap G_\beta)$ are intransitive normal subgroups of $G$, and thus $R\cap G_\beta = 1$ by Lemma \ref{l:semi}(i).
		Moreover, the image of $A$ in $\GL_k(\bbZ/p^2\bbZ)/R\cong\GL_k(p)$ is of order $p$, and hence $A$ also has order $p$ in $\GL_k(\bbZ/p^2\bbZ)$ as $R\cap G_\beta = 1$ and $A\in G_\beta$.
		Now
		\[\begin{aligned}
		I=A^p &= (I+E+pM)^p
		\\&=(I+E)^p+p \sum_{\ell=0}^{p-1}(I+E)^\ell M(I+E)^{p-1-\ell}\\
		&=I+pE+p\sum_{\ell =0}^{p-1}(I+\ell E)M(I+(p-1-\ell)E)\\
		&=I+pE+p\sum_{\ell =0}^{p-1}(M+(p-1-\ell)ME+\ell EM+\ell(p-1-\ell)EME)\\
		&=I+pE+\frac{p^2(p-1)}{2}(ME+EM)+\frac{p^2(p-1)(p-2)}{6}EME.
		\end{aligned}\]
		Since $p\neq 2$ or $3$, we have $2\mid p-1$ and $6\mid (p-1)(p-2)$, so
		\begin{equation*}
		I=I+pE\neq I,
		\end{equation*}
		a contradiction.
		Therefore, $p=2$ or $3$.
	\end{proof}

	\begin{lem}
		\label{l:prime-power_ind}
		Let $Z$ be a finite group. Suppose $Z$ has a subgroup of index $p^\ell$ for some prime $p$, and $Z_2\normeq Z_1\normeq Z$. If $p^\ell$ does not divide $|Z:Z_1|\cdot|Z_2|$, then $Z_1/Z_2$ has a subgroup of index $p^t$ for some $t\geqs 1$.
	\end{lem}
	
	\begin{proof}
		Let $P$ be a subgroup of $Z$ of index $p^\ell$. Note that $(P\cap Z_1)Z_2/Z_2$ has index $p^t$ in $Z_1/Z_2$ for some $t\geqs 0$. If $t = 0$, then $(P\cap Z_1)Z_2 = Z_1$, and so
		\begin{equation*}
		|Z:Z_1|\cdot |Z_2| = \frac{|Z|}{|P\cap Z_1:P\cap Z_2|}.
		\end{equation*}
		This implies that $p^\ell = |Z:P|$ divides $|Z:Z_1|\cdot |Z_2|$, which is incompatible with our assumption. Therefore, $t\geqs 1$ and the proof is complete.
	\end{proof}
	
	In view of Lemma \ref{l:S=T}, now we consider each case listed in Corollary \ref{c:trans_linear_prime-power} in turn.

	\begin{lem}
		\label{l:T_SL}
		Assume Hypothesis \ref{hypo:imprim}, where $G$ is semiprimitive and $G^\calB$ is affine. If $N\cong\bbZ_{p^2}^k$ and $\SL_a(p^f)\normeq G_B^\calB\leqs\mathrm{\Gamma L}_a(p^f)$ for $a\geqs 3$ and $k = af$, then $G$ is not of rank $3$ on $\Omega$.
	\end{lem}
	
	\begin{proof}
		Suppose $G$ has rank $3$ on $\Omega$ and write $q = p^f$. First note that if $G_{B,B'}^\calB$ is soluble, then $k = a = 3$ and $q = p\in\{2,3\}$. With the aid of {\sc Magma}, one can check that $\GL_3(\bbZ/9\bbZ)$ has no subgroup isomorphic to $\SL_3(3)$. And it is also easy to check that if $q = 2$ and $k = 3$, then $G = \bbZ_4^3.\SL_3(2)$ has a unique subgroup isomorphic to $\SL_3(2)$ up to conjugacy, while $G$ is not semiprimitive on the set of cosets of $\SL_3(2)$.
		
		Thus, we may now consider the case where $G_{B,B'}^\calB$ is insoluble. By Lemma \ref{l:semi}(iii), the group $Z:=G_{B,B'}^\calB$ has a subgroup of index $|K| = p^k$. Note that $Z_1 := Z\cap (G_B^\calB)^{\infty}\normeq Z$, which is isomorphic to $Z_2{:}S$ with $Z_2\cong \bbZ_p^{(a-1)f}$ and $S\cong \GL_{a-1}(q)$. By Lemma \ref{l:prime-power_ind}, $S$ has a subgroup of index $p^t$ for some $t\geqs 1$, and with the same reason, the group $\PSL_{a-1}(q)$ has a subgroup of index $p^t$. Now we apply Theorem \ref{t:prime-power}, and thus $f = 1$ and
		\begin{equation*}
		(a,p) \in\{(3,5),(3,7),(3,11),(4,2),(5,2)\}.
		\end{equation*}
		Now Lemma \ref{l:pre-image_p=2,3} eliminates the cases where $(a,p) \in \{(3,5),(3,7),(3,11)\}$. Finally, with the aid of {\sc Magma}, one can check that for $a\in\{4,5\}$, the group $\GL_a(\mathbb{Z}/4\mathbb{Z})$ has no subgroup isomorphic to $\SL_a(2)$.
	\end{proof}
	
	\begin{lem}
		\label{l:T_Sp}
		Assume Hypothesis \ref{hypo:imprim}, where $G$ is semiprimitive and $G^\calB$ is affine. If $N\cong\bbZ_{p^2}^k$ and $\Sp_{2a}(p^f)'\normeq G_B^\calB$ for $a\geqs 2$ and $k = 2af$, then $G$ is not of rank $3$ on $\Omega$.
	\end{lem}
	
	\begin{proof}
		Suppose $G$ has rank $3$ on $\Omega$. For the cases where $(a,p^f) = (2,2)$ or $(2,3)$, one can check using {\sc Magma} that $\GL_4(\bbZ/p^2\bbZ)$ has no subgroup isomorphic to $\Sp_4(p^f)'$, so we may now assume that $(a,p^f)\ne (2,2)$ or $(2,3)$. Once again, we write $q = p^f$ and $Z = G_{B,B'}^\calB$, which has a subgroup of index $p^k$ by Lemma \ref{l:semi}(iii). In this setting, $Z_1:=Z\cap (G_B^\calB)^{\infty}\normeq Z$, which is isomorphic to $Z_2{:}S$ with $|Z_2| = q^{2a-1}$ and $S\cong \bbZ_{q-1}\times \Sp_{2a-2}(q)$. By arguing as in the proof of Lemma \ref{l:T_SL}, the simple group $\PSp_{2a-2}(q)$ has a subgroup of index $p^t$ for some $t\geqs 1$. It follows by Theorem \ref{t:prime-power} that $f = 1$ and
		\begin{equation*}
		(a,p) \in\{(2,5),(2,7),(2,11),(3,2),(3,3)\}.
		\end{equation*}
		Once again, Lemma \ref{l:pre-image_p=2,3} eliminates the cases where $a = 2$ and $p\in\{5,7,11\}$. With the aid of {\sc Magma}, one can also check that $\GL_6(\bbZ/4\bbZ)$ has no subgroup isomorphic to $\Sp_6(2)$, while if $(a,p) = (3,3)$, then $Z$ has no subgroup of index $3^6$, which is incompatible with Lemma \ref{l:semi}(iii).
	\end{proof}
	
	\begin{lem}
		\label{l:T_G2}
		Assume Hypothesis \ref{hypo:imprim}, where $G$ is semiprimitive and $G^\calB$ is affine. If $N\cong\bbZ_{p^2}^k$, $G_2(2^f)'\normeq G_B^\calB$ and $k = 6f$, then $G$ is not of rank $3$ on $\Omega$.
	\end{lem}
	
	\begin{proof}
		Suppose $G$ has rank $3$ on $\Omega$. First note that if $f = 1$, then $G_B^\calB = G_2(2)'$ or $G_2(2)$. If $G_B^\calB$ is the former, then $2^6\nmid |G_B^\calB|$, which is incompatible with Lemma \ref{l:semi}(iii). Thus, $G_B^\calB = G_2(2)$, and so $G = \bbZ_4^6.G_2(2)$, which has rank $6$ on $\Omega$, as can be computed using {\sc Magma}. Now assume $f>1$ and write $Z = G_{B,B'}^\calB$, which has a subgroup of index $2^{6f}$. We use the same technique as above and let $Z_1 = Z\cap (G_B^\calB)^{\infty}\normeq Z$, noting that $Z_1\cong Z_2{:}S$, where $|Z_2| = 2^{5f}$ and $S\cong \GL_2(2^f)$. It follows that the simple group $\SL_2(2^f)$ has a subgroup of index $2^t$ for some $t\geqs 1$. However, this is incompatible with Theorem \ref{t:prime-power}.
	\end{proof}

	\begin{prop}
		\label{p:semi_affine}
		Assume Hypothesis \ref{hypo:imprim}, where $G$ is semiprimitive and $G^\calB$ is affine. If $G$ has rank $3$ on $\Omega$, then $G$ has a regular normal subgroup $N$, where $N$ is a special $p$-group for some prime $p$.
	\end{prop}
	
	\begin{proof}
		Combine Lemmas \ref{l:p=r_affine}, \ref{l:S=T}, \ref{l:T_SL}, \ref{l:T_Sp} and \ref{l:T_G2}.
	\end{proof}
	
	We conclude that the proof of Theorem \ref{thm:semi} is complete by combining Propositions \ref{p:GB_as_semi} and \ref{p:semi_affine}.

	\section{Concluding remarks}
	
	\label{s:rmks}
	
	In this final section, we record some additional remarks and examples on some cases appearing in the statements of our main theorems. Same as above, let $G$ be a finite imprimitive permutation group on $\Omega$ of rank $3$, and let $\calB$ be the unique non-trivial block system of $G$.
	Set $K:=G_{(\calB)}$, $|\calB|=n$ and let $B,B'\in\calB$. The following is a family of genuine rank $3$ groups lying in Theorem \ref{thm:reduction}(C).
	
	\begin{ex}\label{ex:trans}
		Assume Hypothesis \ref{hypo:imprim}, where $Y = V{:}L$ is affine and $n\geqs 3$. Let $D\cong L$ be the diagonal subgroup of $L^n$, and let
		\begin{equation*}
		W = \la (v_1,\dots,v_n)\in V^n:v_1+\cdots +v_n = 0\ra < V^n.
		\end{equation*}
		Consider the group $G = (W{:}L){:}X < Y\wr X$, so $V^n$ is not a subgroup of $G$. As noted in the proof of \cite[Lemma 3.7]{DGLPP_rank3}, $G$ is of rank $3$ and $K_{(B)}$ is transitive on $B'$.
	\end{ex}
	
	Now we consider of Theorem \ref{thm:reduction}(D). Here we assume Hypothesis \ref{hypo:imprim} and let $L = G\cap \soc(Y)^n$.
	
	\begin{lem}
		\label{l:KB>1_intrans}
		Suppose $G$ has rank $3$ on $\Omega$, $C_G(L) = L$ and $K_{(B)}\ne 1$ is intransitive on $B'$. Then $K\cong \bbZ_p^a{:}R$ is a Frobenius group, where $a>d$ and $R$ is cyclic.
	\end{lem}

	\begin{proof}
		First note that $K_{(B)}\normeq K\normeq G_{B'}$, and $G_{B'}^{B'}\leqs\GL_d(p)$ is a transitive linear group. If $K_{(B)}^{B'} = 1$, then $K_{(B)}^{B''} = 1$ for any $B''\in\calB\setminus\{B\}$ by the $2$-transitivity of $G^\calB$, which implies that $K_{(B)}$ fixes every point in $\Omega$ and so $K_{(B)} = 1$, a contradiction. Hence, $K_{(B)}^{B'}$ is a non-trivial intransitive normal subgroup of $K^{B'}$, which yields $K^{B'}\normeq G_{B'}^{B'}$ is an imprimitive normal subgroup. By Theorem \ref{t:norm_reducible}, this implies that $K_B\cong K^{B'}\cong M{:}R$ is a Frobenius group, where $R$ is a cyclic subgroup of $\bbZ_{p^d-1}$.  
		Moreover, $K_{(B)}^{B'}$ is an elementary abelian $p$-group of order less than $p^d$. In particular, $K_{(B)}\leqs L$ is an elementary $p$-group. Note that $K^B=K/K_{(B)}\cong M{:}R$ and $M=L^B=L/(L\cap K_{(B)})$, and hence $K/L\cong R$. Since $p\nmid |R|$, we obtain that $K\cong L{:}R\cong\bbZ_p^a{:}R$ is a Frobenius group for some $a>d$.		
	\end{proof}

	Finally, we give some an example of rank $3$ semiprimitive group with a regular normal subgroup $N$. By Theorem \ref{thm:semi}, $N$ is a special $p$-group for some prime $p$, and moreover, $N$ is a group of type (i), (iv) or (v) of Lemma \ref{l:regularnormal}. As discussed in Section \ref{s:intro}, we will work on a classification of such groups in future work.
	The following example gives a genuine rank $3$ group arising in type~(i).

	\begin{ex}
		\label{ex:affine}
		Define the following two subgroups of $\GL_4(2)$:
		\[\begin{aligned}
		H_1&=\left\langle\left(\begin{array}{cccc}1&1&0&0\\ 0&1&0&0\\0&0&1&1\\0&0&0&1\end{array}\right),\left(\begin{array}{cccc}1&0&0&0\\ 0&0&0&1\\0&1&0&0\\ 0&0&1&0\end{array}\right)\right\rangle\cong\GL_3(2), \\
		H_2&=\left\langle\left(\begin{array}{cccc}1&1&0&0\\0&1&0&0\\0&0&1&1\\ 0&0&0&1\end{array}\right),\left(\begin{array}{cccc}0&0&1&0\\1&0&0&0\\0&1&0&0\\ 0&0&0&1\end{array}\right)\right\rangle\cong\GL_3(2).
		\end{aligned}\]	
		Note that $H_1$ fixes a hyperplane and $H_2$ fixes a $1$-dimensional subspace of $\mathbb{F}_2^4$.
		Write $G_1:=\mathbb{F}_2^4{:}H_1$ and $G_2:=\mathbb{F}_2^4{:}H_2$. Then both $G_1$ and $G_2$ are imprimitive groups acting on the vector space $V=\mathbb{F}_2^4$. As can be checked using {\sc Magma}, both $G_1$ and $G_2$ are of rank $3$, but only $G_2$ is semiprimitive.
	\end{ex}



\end{document}